\documentclass{article}

\usepackage{amsmath}
\usepackage{amsthm}
\usepackage{amssymb}
\usepackage[all]{xy}
\usepackage{enumitem}

\newcommand{\Qp}{\mathbb{Q}_p}
\newcommand{\qp}{\Qp}
\newcommand{\Q}{\mathbb{Q}}
\newcommand{\Qpbar}{\overline{\mathbb{Q}}_p}
\newcommand{\qpbar}{\Qpbar}
\newcommand{\Z}{\mathbb{Z}}
\newcommand{\Zp}{\mathbb{Z}_p}
\newcommand{\zp}{\Zp}
\newcommand{\Zpbar}{\overline{\Z}_p}
\newcommand{\Fpbar}{\overline{\mathbb{F}}_p}
\newcommand{\fpbar}{\Fpbar}
\newcommand{\Fp}{\mathbb{F}_p}
\newcommand{\Oe}{\mathcal{O}_E}
\newcommand{\F}{\mathbb{F}}
\newcommand{\ef}{\F}
\newcommand{\cC}{\mathcal{C}}
\newcommand{\sq}{\square}
\newcommand{\rsq}{R^{\sq}}
\newcommand{\wtau}{\widetilde{\tau}}
\newcommand{\wS}{\widetilde{S}}
\newcommand{\wI}{\widetilde{I}}
\newcommand{\ofp}{\cO_{F,p}}
\newcommand{\cO}{\mathcal{O}}
\newcommand{\A}{\mathbb{A}}
\newcommand{\res}{\mathrm{res}}
\newcommand{\m}{\mathfrak{m}}
\newcommand{\N}{\mathbb{N}}
\newcommand{\cS}{\mathcal{S}}
\renewcommand{\a}{\mathfrak{a}}
\newcommand{\lam}{\lambda}
\newcommand{\tv}{\tilde{v}}
\renewcommand{\d}{\mathrm{d}}
\newcommand{\T}{\mathbb{T}}
\newcommand{\p}{\mathfrak{p}}
\newcommand{\C}{\mathbb{C}}
\newcommand{\Spl}{\mathrm{Spl}}
\newcommand{\cH}{\mathcal{H}}

\newcommand{\rhobar}{\bar{\rho}}
\newcommand{\GQp}{G_{\Qp}}
\newcommand{\ord}{\mathrm{ord}}
\newcommand{\ab}{\mathrm{ab}}
\newcommand{\x}{^{\times}}
\newcommand{\Fil}{\mathrm{Fil}}
\newcommand{\fil}{\Fil}
\newcommand{\Frob}{\mathrm{Frob}}
\newcommand{\univ}{\mathrm{univ}}

\newcommand{\into}{\hookrightarrow}
\newcommand{\onto}{\twoheadrightarrow}

\DeclareMathOperator{\GL}{GL}
\DeclareMathOperator{\Art}{Art}
\DeclareMathOperator{\HT}{HT}
\DeclareMathOperator{\Gal}{Gal}
\DeclareMathOperator{\Hom}{Hom}

\DeclareMathOperator{\End}{End}
\DeclareMathOperator{\ad}{ad}
\DeclareMathOperator{\pr}{pr}

\theoremstyle{plain} 
\newtheorem{lemma}[equation]{Lemma}
\newtheorem{proposition}[equation]{Proposition}
\newtheorem{theorem}[equation]{Theorem}
\newtheorem{corollary}[equation]{Corollary}

\theoremstyle{definition}
\newtheorem{definition}[equation]{Definition}

\theoremstyle{remark}
\newtheorem{remark}[equation]{Remark}

\numberwithin{equation}{subsection}

\title{Multiplicities in the ordinary part of mod $p$ cohomology for $\GL_n(\Qp)$}
\author{John Enns\footnote{The author was supported by an NSERC grant while this work was carried out.}} 
\date{}

\begin{document}
\maketitle 

\begin{abstract}
Given a continuous ordinary Galois representation $\rhobar:\GQp\rightarrow\GL_n(\Fpbar)$, Breuil and Herzig constructed an admissible smooth $\Fpbar$-representation $\Pi(\rhobar)^{\ord}$ of $\GL_n(\Qp)$ and showed that it occurs in certain globally defined mod $p$ cohomology spaces in \cite{BH}. By applying Taylor-Wiles patching to spaces of ordinary automorphic representations we prove that the indecomposable pieces of $\Pi(\rhobar)^{\ord}$ each occur with the same multiplicity at a well-chosen tame level.
\end{abstract}

Let $\bar{\rho}:G_{\Qp}\rightarrow\GL_n(\Fpbar)$ be a continuous local mod $p$ Galois representation. It is hoped that there will be a mod $p$ local Langlands correspondence $\rhobar\mapsto \Pi(\rhobar)$ associating with $\rhobar$ a smooth admissible $\Fpbar$-representation of $\GL_n(\Qp)$ compatible with the mod $p$ cohomology of (globally defined) arithmetic manifolds. Since many aspects of the local correspondence remain mysterious (see \cite{BreICM} for an overview of what is known) and in light of this expected local-global compatibility, much of the work done so far on the topic has involved studying globally defined representations of $\GL_n(\Qp)$. This paper continues this theme.

Let $F$ be a CM number field in which $p$ is totally split, having maximal totally real subfield $F^+$. Suppose that $\bar{r}:G_F\rightarrow \GL_n(\Fpbar)$ is a global mod $p$ Galois representation arising from an automorphic representation of a definite unitary group $G_{/F^+}$ that becomes isomorphic to $\GL_n$ over $F$. Let $S_p$ denote the set of places of $F^+$ dividing $p$, and fix a choice of place $\tv$ of $F$ dividing $v$ for each $v\in S_p$. In Section \ref{subsec:aut forms} it is explained how in this situation the space of mod $p$ automorphic forms on $G$ gives rise to an admissible $\Fpbar$-representation of $\prod_{v\in S_p}\GL_n(F_{\tv})$, which we denote $S(U^p,\F)[\m_{\bar{r}}^{\Sigma}]$ depending on a choice of tame level $U^p\leq G(\A_{F^+}^{\infty,p})$ and an auxiliary set of places $\Sigma$. It is hoped that this space will realize the local mod $p$ Langlands correspondence for $\bar{r}|_{G_{F_{\tv}}}$, $v\in S_p$. 

We are a long way from understanding the full representation $S(U^p,\F)[\m_{\bar{r}}^{\Sigma}]$, with the most pressing problem being a lack of understanding of the supersingular representations of $\GL_n(F_{\tv})$. In fact, only recently has there been significant progress in understanding its $\prod_{v\in S_p}\GL_n(\mathcal{O}_{F_{\tv}})$-socle, which is the modern formulation of the weight part of Serre's conjecture (see \cite{LLHLM} and other papers by those authors for the most recent work). 

When $\bar{r}$ is ordinary (meaning upper-triangular) at places dividing $p$, the ordinary subrepresentation $S(U^p,\F)[\m_{\bar{r}}^{\Sigma}]^{\ord}$ is nonzero and is studied in \cite{BH}. It is defined to be the largest subrepresentation of $S(U^p,\F)[\m_{\bar{r}}^{\Sigma}]$ whose irreducible subquotients are all subquotients of principal series representations of $\prod_{v\in S_p} \GL_n(F_{\tv})$. Given any ordinary local Galois representation $\rhobar$, \cite{BH} defines a representation of $\GL_n(\Qp)$ called $\Pi(\rhobar)^{\ord}$ (the point being it depends only on $\rhobar$), and one of their main results in the global situation at hand is an inclusion
\begin{equation}\label{eq:inclusion}
\bigotimes_{v\in S_p} \Pi(\bar{r}|_{G_{F_{\tv}}})^{\ord}\hookrightarrow S(U^p,\F)[\m_{\bar{r}}^ {\Sigma}]^{\ord}.
\end{equation}
In fact, by taking into account certain unknown multiplicities $d_{\sigma}\geq1$ of the indecomposable pieces of each $\Pi(\bar{r}|_{G_{F_{\tv}}})^{\ord}$, they obtain an \emph{essential} inclusion in (\ref{eq:inclusion}). These multiplicities are indexed by the set of ordinary Serre weights $\sigma$, which is to say the constituents of the $\prod_{v\in S_p}\GL_n(\mathcal{O}_{F_{\tv}})$-socle of $S(U^p,\F)[\m_{\bar{r}}^{\Sigma}]^{\ord}$, and are strongly related to the multiplicities with which the $\sigma$ appear in this socle. The precise statement of this result is Theorem 4.4.7 in \cite{BH}, also summarized in Section \ref{sec:application} below.

The main goal of this paper (Theorem \ref{thm:application}) is to improve the result of \cite{BH} by showing that $d_{\sigma}=d$ is independent of $\sigma$ for a fixed well-chosen tame level $U^p$. To do this we show that all the ordinary Serre weights $\sigma$ appear in $S(U^p,\F)[\m_{\bar{r}}^{\Sigma}]^{\ord}$ with the same multiplicity (Theorem \ref{thm:main result}). This may be viewed as a step towards the structure of the mod $p$ local Langlands correspondence.

The main ingredients are a strong link (previously known) between the mod $p$ ordinary subspace and spaces of ordinary automorphic forms in characteristic $0$, and Taylor-Wiles patching of spaces of ordinary automorphic representations. In fact we use an adaptation of Diamond's variant \cite{Dia} of the Taylor-Wiles method, to show freeness over a Hecke algebra in the proof of Theorem \ref{thm:main result}. Results on multiplicities of Serre weights have been obtained previously using the formalism of patching functors such as those developed in \cite{CEG+} but these techniques require an understanding of the geometry of potentially crystalline deformation rings which is the subject of the Breuil-M\'ezard conjecture and is lacking in general. The novelty of patching spaces of \emph{ordinary} automorphic representations only is that we can use ordinary crystalline deformation rings (\cite{Ger}) at $p$, which are better understood. In fact, since we are only interested in generic Galois representations, it is enough to use the simple ``naive'' ordinary deformation rings of \cite{CHT}. 

The paper is structured as follows: in Section \ref{sec:ordinary lifting rings} we recall results on ordinary deformation rings at places above $p$. In Section \ref{sec:ordinary hecke algebras} we review results on ordinary automorphic representations of $G$. In Section \ref{sec:patching} we set up the patching data and show the main results.

\subsection{Acknowledgements}

We thank Florian Herzig for asking about the main result of this note, and for giving some initial suggestions and helpful comments on an earlier draft. 

It will be clear to experts that the patching procedure used in this article draws heavily from the work of others, mainly \cite{Ger}, \cite{CEG+}, \cite{EG} and \cite{Tho}. 

\subsection{Notation and conventions}\label{sec:notation}

If $K/\Qp$ is a finite extension, $G_K$ denotes a fixed choice of absolute Galois group with inertia subgroup $I_K$, and $\Art_K:K\x\rightarrow G_K^{\ab}$ denotes the local reciprocity map of class field theory, normalized so that uniformizers correspond to geometric Frobenius elements. 

Throughout, $E$ denotes a finite extension of $\Qp$ that serves as a field of coefficients for our representations. It has ring of integers $\Oe$ and residue field $\F$. We also fix a choice of algebraic closure $\Qpbar\supset E$ having ring of integers $\Zpbar$ and residue field $\Fpbar$. We always assume that $E$ is large enough to contain the image of every embedding $K\into \Qpbar$. For the purposes of defining deformation problems, $\cC_{\Oe}$ denotes the category of local noetherian $\Oe$-algebras having residue field equal to $\F$. Given a continuous residual Galois representation $\rhobar:G_K\rightarrow\GL_n(\F)$ we write $\rhobar^{\sq}:G_K\rightarrow\GL_n(R^{\sq}_{\rhobar})$ for the universal (framed) lift of $\rhobar$ to an object of $\cC_{\Oe}$. In fact we usually write $R^{\sq}$ in place of $R^{\sq}_{\rhobar}$ when $\rhobar$ is clear from the context.

If $\rho:G_K\rightarrow\GL_n(B)$ is a continuous representation, where $B$ is a finite-dimensional local $E$-algebra, we have $K\otimes_{\Qp}E=\prod_{\tau:K\into E}E$ and we use this decomposition to define the set $\HT_{\tau}$ of Hodge-Tate weights of $\rho$ with respect to the embedding $\tau$. We normalize the definition of Hodge-Tate weights so that the cyclotomic character $\epsilon$ has weight $\{-1\}$. Given an element $\lambda\in(\mathbb{Z}_{+}^{n})^{\mathrm{Hom}(K,E)}$ where $\Z^n_+\subset \Z^n$ denotes the set of nonincreasing $n$-tuples of integers, we say that a representation $\rho:G_{K}\rightarrow\GL_n(B)$ has Hodge type $\lambda$ if $\mathrm{HT}_{\tau}=\{\lambda_{\tau,i}+n-i\}_{1\leq i\leq n}$.

If $M$ is a finite free $\Oe$-module then $M^{\d}$ stands for the $\Oe$-module $\Hom_{\Oe}(M,\Oe)$. This is a very special case of the Schikhof duality used in the patching of \cite{CEG+}. On the other hand, if $V$ is a vector space over a field then we write $V^{\vee}$ for the dual vector space. For example, we would write $M^{\d}\otimes_{\Oe}\F=(M\otimes_{\Oe}\F)^{\vee}$. 

\section{Ordinary lifting rings}\label{sec:ordinary lifting rings}

The goal of this section is to recall the ``naive'' ordinary lifting rings of \cite{CHT} and collect some of their properties. Until the end of the section let $K$ be a finite extension of $\Qp$.

\subsection{Ordinary crystalline representations}\label{subsec:ordinary crys repns}

\begin{definition}\label{def:crystalline characters}
Let $\lambda\in(\mathbb{Z}_{+}^{n})^{\mathrm{Hom}(K,E)}$. We define associated with $\lambda$ an $n$-tuple of characters $\{\chi_{i}^{\lambda}\}_{1\leq i\leq n}:I_{K}\rightarrow\mathcal{O}_{E}^{\times}$ by setting 
\begin{equation*}
\chi_{i}^{\lambda}=\epsilon^{-i+1}\prod_{\tau:K\hookrightarrow E}(\tau\circ\mathrm{Art}_{K}^{-1})^{-\lambda_{\tau,n+1-i}}.
\end{equation*}
\end{definition}
A character $\psi:G_{K}\rightarrow E^{\times}$ is crystalline of Hodge type $\lambda_i$ for $1\leq i\leq n$ if and only if it agrees with $\chi_{n+1-i}^{\lambda}$ on $I_{K}$. We say that $\lambda\in(\mathbb{Z}_{+}^{n})^{\mathrm{Hom}(K,E)}$ is \emph{special} if for each $1\leq i\leq n-1$ there exists $\tau:K\hookrightarrow E$ such that $\lambda_{\tau,i}>\lambda_{\tau,i+1}$. We define $\lambda$ to be \emph{generic} if for each $\tau$ and $1\leq i\leq n-1$ we have $\lambda_{\tau,i}-\lambda_{\tau,i+1}>1$.

\begin{definition}\label{def:ordinary galois representation}
Let $E'$ be a finite extension of $E$ and let $\lambda\in(\mathbb{Z}_{+}^{n})^{\mathrm{Hom}(K,E)}$. A continuous homomorphism $\rho:G_{K}\rightarrow\mathrm{GL}_{n}(E')$ is called \emph{ordinary of weight $\lambda$} if it is conjugate to a homomorphism of the form
\begin{equation*}
 \begin{pmatrix}\psi_{1} & * & \cdots & * & *\\
 & \psi_{2} & \cdots & * & *\\
 &  & \ddots & \vdots & \vdots\\
 &  &  & \psi_{n-1} & *\\
 &  &  &  & \psi_{n}
\end{pmatrix}
\end{equation*}
 where $\psi_{i}|_{I_K}=\chi_{i}^{\lambda}$.
\end{definition}

We remark that Lemma 3.1.4 of \cite{GG} shows that if $\rho$ is ordinary of weight $\lambda$ and $\lambda$ is special then $\rho$ is crystalline of Hodge type $\lambda$. 

\subsection{Naive ordinary crystalline lifting rings}\label{subsec:naive lifting rings}

In the case of a generic residual representation, the naive ordinary deformation problem that one might write down is representable and one gets a good ordinary deformation ring. This is explained in Section 2.4 of \cite{CHT} and we now review some of the results from there. Let $\bar{\rho}:G_{K}\rightarrow\mathrm{GL}_{n}(\F)$ be a continuous homomorphism conjugate to one of the form
\begin{equation}\label{eq:rhobar}
\begin{pmatrix}\bar{\psi}_{1} & * & \cdots & * & *\\
 & \bar{\psi}_{2} & \cdots & * & *\\
 &  & \ddots & \vdots & \vdots\\
 &  &  & \bar{\psi}_{n-1} & *\\
 &  &  &  & \bar{\psi}_{n}
\end{pmatrix}
\end{equation}
where the characters $\bar{\psi}_{i}:G_{K}\rightarrow\F\x$ are distinct on $I_K$. This means that $\F^{n}$ has a decreasing $\bar{\rho}$-invariant filtration by subspaces $\F^n=\overline{\Fil}^{n}\supset\overline{\Fil}^{n-1}\supset\cdots\supset\overline{\Fil}^{1}\supset0$ such that $\overline{\fil}^{i}/\overline{\fil}^{i-1}=\ef(\bar{\psi}_{i})$ for $1\leq i\leq n$, and because the $\bar{\psi}_{i}$ are distinct this filtration is unique. 

Fix any choice of characters $\psi_{i}:I_{K}\rightarrow\Oe^{\times}$ lifting $\bar{\psi}_{i}|_{I_K}$ for $1\leq i\leq n$. Let $\mathcal{D}$ denote the set of all lifts of $\bar{\rho}$ to continuous $\rho:G_{K}\rightarrow\GL_n(A)$ with $A$ in $\cC_{\Oe}$ such that $A^{n}$ has a decreasing filtration $A^{n}=\fil^{n}\supset\fil^{n-1}\supset\cdots\supset\fil^{1}\supset0$ by $A[G_{K}]$-submodules which are $A$-direct summands such that $\fil^{i}/\mathfrak{m}_{A}=\overline{\fil}^{i}$ and $I_{K}$ acts on $\Fil^i/\Fil^{i-1}$ by the character $\psi_{i}$ for $1\leq i\leq n$. Again, since the $\psi_{i}$ are distinct such a choice of filtration is unique and using this one shows easily that $\mathcal{D}$ is a deformation problem in the sense of \cite{CHT}, Definition 2.2.2. It follows from Lemma 2.2.3 \emph{loc. cit.} that there exists a quotient $R(\psi)$ of $R^{\sq}$ such that a morphism $\zeta:R^{\sq}\rightarrow A$ in $\mathcal{C}_{\Oe}$ factors through $R(\psi)$ iff $\zeta\circ\rho^{\square}$ belongs to $\mathcal{D}$. 

This construction is compatible with finite extension of the coefficient field in the following sense. Let $E'/E$ be a finite extension. Then it is well known that $\rsq\otimes_{\Oe}\mathcal{O}_{E'}=R_{\bar{\rho}\otimes_{\ef}\ef'}^{\square}$, that is, the universal lifting ring of $\bar{\rho}\otimes_{\ef}\ef'$ for the category $\mathcal{C}_{\mathcal{O}_{E'}}$. One checks that $R(\psi)\otimes_{\Oe}\mathcal{O}_{E'}=R(\psi\otimes_{\Oe}\mathcal{O}_{E'})$.

\begin{proposition}\label{prop:properties of ordinary deformation rings} The ring $R(\psi)$ has the following properties:
\begin{enumerate}
\item Let $E'$ be a finite extension of $E$. An $\Oe$-algebra homomorphism $\zeta:R^{\square}\rightarrow E'$ factors through $R(\psi)$ iff $E'^{n}$ has a decreasing $\zeta\circ\rho^{\square}$-invariant filtration  $E'^{n}=\fil^{n}\supset\fil^{n-1}\supset\cdots\supset\fil^{1}\supset0$ such that $\Fil^i/\Fil^{i-1}$ is the character $\psi_i\otimes_EE'$ as a representation of $I_{K}$.

\item If in addition we assume that the diagonal characters of $\bar{\rho}$ obey $\bar{\psi}_{i}\bar{\psi}_{j}^{-1}\neq1,\bar{\epsilon}$ for each $i<j$, $R(\psi)$ is a power series ring in $n^{2}+[K:\Qp]\frac{n(n-1)}{2}$ variables over $\Oe$. 
\end{enumerate}
\end{proposition}

\begin{proof}
The first claim follows from the definition of $R(\psi)$ and the aforementioned compatibility with finite extension, as well as the fact that any such homomorphism $\zeta$ lands in $\mathcal{O}_{E'}$ and any such filtration on $E'^{n}$ descends to one on $\mathcal{O}_{E'}^{n}$. The second claim is Lemmas 2.4.7 and 2.4.8 of \cite{CHT}. 
\end{proof}

In the special case that $\psi_{i}=\chi_{i}^{\lambda}$ for some $\lambda\in(\mathbb{Z}_{+}^{n})^{\mathrm{Hom}(K,E)}$ we write $R^{\triangle_{\lambda}}$ instead of $R(\psi)$. Thus by the previous proposition $R^{\triangle_{\lambda}}$ is formally smooth if $\lambda$ is generic.

\begin{definition}\label{def:inertially generic}
Let $\rhobar:G_K\rightarrow\GL_n(\F)$ be conjugate to an upper triangular representation as in (\ref{eq:rhobar}). We call such a representation \emph{ordinary}. We say that $\rhobar$ is \emph{inertially generic} if for each $1\leq i\neq j\leq n$ we have
\begin{equation*}
\bar{\psi}_i\bar{\psi}_j^{-1}|_{I_K}\neq 1,\bar{\epsilon}^{\pm1},\bar{\epsilon}^{\pm2}.
\end{equation*}
\end{definition}

\begin{remark}
This is stronger than the definition of inertially generic in \cite{BH}.
\end{remark}

\noindent The existence of an inertially generic representation implies that $p\geq 5n-5$.

\section{Ordinary automorphic forms and Hecke algebras}\label{sec:ordinary hecke algebras}

In Sections \ref{subsec:repns of GLn} and \ref{subsec:aut forms} we recall the necessary background about spaces of automorphic forms on definite unitary groups. In Section \ref{subsec:aut gal repns} we review the necessary facts about spaces of ordinary forms.

\subsection{Representations of $\GL_n$}\label{subsec:repns of GLn}

Let $B_{n}$ denote the standard upper-triangular Borel subgroup of the algebraic group $\GL_n$ and $T_{n}$ the diagonal torus. We identify the Weyl group $W$ of $T_{n}$ with $S_{n}$ and the cocharacter group of $T_{n}$ with $\mathbb{Z}^{n}$ in the usual way. If $\lambda\in\mathbb{Z}_{+}^{n}$ is a dominant weight we write $M(\lambda)$ for the dual Weyl module of $\mathrm{GL}_{n,\mathbb{Z}}$ of highest weight $\lambda$, defined as the algebraic induction
\begin{equation*}
M(\lambda)=\mathrm{Ind}_{B_{n}}^{\mathrm{GL}_{n}}(w_{0}\lambda)_{\mathbb{Z}}
\end{equation*}
where $w_{0}\in W$ is the longest element. Then for any ring $A$ we write $M(\lambda)_{A}:=M(\lambda)\otimes A$, which is isomorphic to $\mathrm{Ind}_{B_{n}}^{\mathrm{GL}_{n}}(w_{0}\lambda)_{A}$ by \cite{RAGS}, II.8.8(1). In particular $M(\lambda)_{A}$ carries an action of $\GL_n(A)$ and is free over $A$.

A \emph{Serre weight of $\mathrm{GL}_{n}(\mathbb{F}_{p})$} is an absolutely irreducible $\ef$-representation of $\mathrm{GL}_{n}(\mathbb{F}_{p})$ or equivalently of $\mathrm{GL}_{n}(\mathbb{Z}_{p})$. Serre weights are classified by the $p$-restricted set of dominant weights $\mathbb{Z}_{+,\mathrm{res}}^{n}=\{\lambda\in\mathbb{Z}^{n}\,|\,\lambda_{i}-\lambda_{i+1}\in[0,p-1]\,\mathrm{for}\,1\leq i\leq n-1\}$: if $\lambda\in\mathbb{Z}_{+,\mathrm{res}}^{n}$ then the sub-$\mathrm{GL}_{n}(\mathbb{F}_{p})$ representation of $M(\lambda)_{\ef_{p}}$ generated by the highest weight vector is absolutely irreducible and we denote this representation $F_{\lambda}$. Each absolutely irreducible representation of $\mathrm{GL}_{n}(\mathbb{F}_{p})$ arises this way and $F_{\lambda}\cong F_{\lambda'}$ iff $\lambda-\lambda'\in(p-1,\ldots,p-1)\mathbb{Z}$ (in particular all Serre weights are defined over $\mathbb{F}_{p}$ and the choice of coefficient field $\ef$ in the definition does not matter). By the remarks above, we may equivalently view $F_{\lambda}$ as the sub-$\mathrm{GL}_{n}(\mathbb{Z}_{p})$ representation of $M(\lambda)_{\Oe}\otimes_{\Oe}\ef$ generated by the highest weight vector, where we think of $\mathrm{GL}_{n}(\mathbb{Z}_{p})$ as a subgroup of $\mathrm{GL}_{n}(\Oe)$ acting on $M(\lambda)_{\Oe}$.

A Serre weight $\sigma$ gives rise to a Hecke algebra 
\begin{equation*}
\mathcal{H}(\sigma):=\mathrm{End}_{\mathrm{GL}_{n}(\qp)}\left(\mathrm{cInd}_{\mathrm{GL}_{n}(\mathbb{Z}_{p})}^{\mathrm{GL}_{n}(\Qp)}(\sigma)\right)
\end{equation*}
where $\mathrm{cInd}$ denotes compact induction. See the beginning of Section 4.3 of \cite{BH} for more details and references concerning what follows.  If $\Pi$ is a smooth $\ef$-representation of $\mathrm{GL}_{n}(\Qp)$ then $\mathcal{H}(\sigma)$ acts on $\mathrm{Hom}_{\mathrm{GL}_{n}(\mathbb{Z}_{p})}(\sigma,\Pi)$ via Frobenius reciprocity. Moreover there is a natural identification of $\mathcal{H}(\sigma)$ with the $\ef$-algebra of functions $\varphi:\mathrm{GL}_{n}(\qp)\rightarrow\mathrm{End}_{\ef}(\sigma)$ such that $\varphi(k_{1}gk_{2})=k_{1}\circ\varphi(g)\circ k_{2}$ for $k_{1},k_{2}\in\mathrm{GL}_{n}(\zp)$ and $g\in\mathrm{GL}_{n}(\Qp)$ with a convolution product. Concretely, in the case at hand there is an isomorphism of $\mathcal{H}(\sigma)$ with the polynomial algebra $\ef[T_{\sigma,1},\ldots,T_{\sigma,n-1},T_{\sigma,n}^{\pm1}]$ where $T_{\sigma,j}$ corresponds to the function having support 
\begin{equation*}
\mathrm{GL}_{n}(\zp)\begin{pmatrix}1_{n-j} & 0\\
0 & p1_{j}
\end{pmatrix}\mathrm{GL}_{n}(\zp) 
\end{equation*} 
taking $\begin{pmatrix}1_{n-j} & 0\\0 & p1_{j}\end{pmatrix}$ to the endomorphism
\begin{equation*}
\sigma\twoheadrightarrow\sigma_{\overline{U}_j(\Fp)}\xleftarrow{\sim} \sigma^{U_j(\Fp)}\hookrightarrow\sigma
\end{equation*}
where $U_{j}$ denotes the unipotent radical of the standard parabolic subgroup of $\GL_n$ corresponding to the partition $n=(n-j)+j$ and $\overline{U}_{j}$ denotes the opposite unipotent radical. We write $\mathrm{Hom}_{\mathrm{GL}_{n}(\mathcal{O}_{F_{\tilde{v}}})}(\sigma,\Pi)^{\mathrm{ord}}$ for the largest subspace of $\mathrm{Hom}_{\mathrm{GL}_{n}(\Zp)}(\sigma,\Pi)$ preserved by $\mathcal{H}(\sigma)$ on which the operators $T_{\sigma,1},\ldots,T_{\sigma,n}$ all act invertibly. This will be referred to as the \emph{ordinary subspace}.

\subsection{Automorphic forms on definite unitary groups}\label{subsec:aut forms}

Let $F$ be an imaginary CM field with maximal totally real subfield $F^+$ and let $c$ denote the nontrivial element of $\Gal(F/F^+)$. Assume that $F/F^{+}$ is unramified at all finite places, and that $4|n[F^{+}:\mathbb{Q}]$. Then we can find an algebraic group $G$ defined over $F^{+}$ such that $G$ is an outer form of $\mathrm{GL}_{n}$ becoming isomorphic to $\mathrm{GL}_{n}$ over $F$, $G$ is quasisplit at each finite place of $F^+$, and for each infinite place $v$ of $F^{+}$, $G(F_{v}^{+})\cong U_{n}(\mathbb{R})$ is compact. It is also possible to fix a reductive model for $G$ over $\mathcal{O}_{F^{+}}$ which we again denote $G$. For each place $v$ of $F^{+}$ that splits as $ww^{c}$ in $F$ there is an isomorphism $\iota_{w}:G(F_{v}^{+})\xrightarrow{\sim}\mathrm{GL}_{n}(F_{w})$ which restricts to an isomorphism $G(\mathcal{O}_{F_{v}^{+}})\xrightarrow{\sim}\mathrm{GL}_{n}(\mathcal{O}_{F_{w}})$. For these statements see e.g. \cite{Tho}, Section 6.

We from now on assume that $p$ is totally split in $F$. Fix a choice of place $\tilde{v}$ of $F$ for each $v|p$. Let $S_{p}$ denote the set of places of $F^{+}$ dividing $p$, and $\wS_{p}$ the chosen set of lifts. Because $p$ is totally split in $F$, there is a bijection between field embeddings $F\hookrightarrow\Qpbar$ and places of $F$ dividing $p$ (and similarly for $F^{+}$). Let $I_p:=\Hom_{\Q}(F^+,\Qpbar)$ and write $\widetilde{I}_p\subset \Hom_{\Q}(F,\Qpbar)$ for the set of embeddings corresponding to $\widetilde{S}_p$. We write $\wtau(\tv)$ for the embedding corresponding to a place $\tv\in\widetilde{S}_p$. Given $\lam\in(\mathbb{Z}_{+}^{n})^{\wI{p}}$ we define a finite free $\Oe$-module $M_{\lambda}$ with an action of $G(\mathcal{O}_{F^{+},p})=\prod_{v\in S_{p}}G(\mathcal{O}_{F_{v}^{+}})$ by 
\begin{equation*}
M_{\lambda}=\bigotimes_{\wtau\in\wI_p}M(\lambda_{\wtau})_{\Oe}
\end{equation*}
where $M({\lambda_{\wtau}})_{\Oe}$ is the dual Weyl module of Section \ref{subsec:repns of GLn} on which we give an action of $G(\cO_{F_v^+})$ obtained via the map $G(\mathcal{O}_{F_{v}^{+}})\cong\mathrm{GL}_{n}(\mathcal{O}_{F_{\tilde{v}}})\hookrightarrow\mathrm{GL}_{n}(\Oe)$ induced by the embedding $\wtau$ that corresponds to $\tilde{v}$. 

A \emph{Serre weight for $G$} is an absolutely irreducible $\ef$-representation of $G(\mathcal{O}_{F^{+},p})$. As in Section \ref{subsec:repns of GLn} the choice of coefficient field does not matter in the definition. If $\lambda\in(\mathbb{Z}_{+,\mathrm{res}}^{n})^{\wI_{p}}$ is $p$-restricted then the facts recalled in Section \ref{subsec:repns of GLn} imply that the action of $G(\mathcal{O}_{F^{+},p})$ on $M_{\lambda}\otimes_{\Oe}\ef$ factors through $\prod_{v\in S_{p}}G(k_{v})$ and $\mathrm{soc}_{G(\ofp)}\left(M_{\lambda}\otimes_{\Oe}\ef\right)=\bigotimes_{\wtau\in \wI_p}F_{\lambda_{\wtau}}$ is absolutely irreducible. We denote this Serre weight by $F_{\lambda}$. Every Serre weight for $G$ arises in this fashion.

Let $U\leq G(\mathbb{A}_{F^{+}}^{\infty,p})\times G(\mathcal{O}_{F^{+},p})$ be a compact open subgroup. If $W$ is an $\Oe$-module with a linear action of $U_p$ where $U_p$ denotes the projection of $U$ to $G(\mathcal{O}_{F^{+},p})$, we let $S(U,W)$ denote the space of \emph{algebraic automorphic forms on $G$ of weight $W$ and level $U$}, which is defined to be the $\Oe$-module of functions $f:G(F^{+})\backslash G(\mathbb{A}_{F^{+}}^{\infty})\rightarrow W$ obeying $f(gu)=u_{p}^{-1}\cdot f(g)$ for all $g\in G(\mathbb{A}_{F^{+}}^{\infty})$ and $u\in U$, where $u_{p}$ is the image of $u$ in $U_p$. By finiteness of class numbers, we may choose finitely many $t_i\in G(\A_{F^+}^{\infty})$, $1\leq i\leq N$ such that $G(\A_{F^+}^{\infty})=\bigsqcup_{i=1}^N G(F^+)t_i U$ and one easily shows there is an isomorphism of $\Oe$-modules
\begin{align}\label{eq:S(U,W)}
S(U,W)&\xrightarrow{\sim} \prod_{i=1}^N W^{t_i^{-1}G(F^+)t_i\cap U} \\
f&\mapsto (f(t_i))_{i=1}^N. \nonumber
\end{align}
In particular, $S(U,W)$ is finitely generated (resp. free) over $\Oe$ whenever $W$ is. We say that $U$ is \emph{sufficiently small} if for some finite place $v$ of $F^+$, the projection of $U$ to $G(F_v^+)$ contains no elements of exact order $p$. This implies that $U$ is sufficiently small in the sense of \cite{Tho}, which is to say that for each $t\in G(\A_{F^+}^{\infty})$ the group $t^{-1}G(F^+)t\cap U$ has no elements of exact order $p$ and one deduces easily from the isomorphism above that if $W$ is free over $\Oe$ and $U$ is sufficiently small then the functor $A\mapsto S(U,W\otimes_{\Oe}A)$ on $\Oe$-modules is exact. See Lemma 6.3 of \cite{Tho}. We will use the following consequence: if $W$ is free over $\Oe$ and $U$ is sufficiently small then the natural map $S(U,W)\otimes_{\Oe}\F\rightarrow S(U,W\otimes_{\Oe}\F)$ is an isomorphism. Also note that if $W$ is free over $\Oe$ then $S(U,W\otimes_{\Oe}E)\cong S(U,W)\otimes E$ without any hypotheses on $U$. This follows directly from (\ref{eq:S(U,W)}).

Next we recall the relationship between classical automorphic forms on $G$ and the algebraic automorphic forms above, following \cite{Tho}. Fix a weight $\lambda\in(\mathbb{Z}^n_{+})^{\wI_{p}}$. Let $S_{\lambda}(\qpbar)=\underrightarrow{\lim}\,S(U,M_{\lambda}\otimes_{\Oe}\qpbar)$ where the limit is with respect to compact open subgroups $U\leq G(\mathbb{A}_{F^{+}}^{\infty,p})\times G(\mathcal{O}_{F^{+},p})$. This space has a natural action of $G(\mathbb{A}_{F^{+}}^{\infty})$ via $g\cdot f(*)=g_{p}\cdot f(*g)$. Fix an isomorphism $\iota:\qpbar\rightarrow\mathbb{C}$ and let $\mathcal{A}$ denote the space of classical automorphic forms on $G(F^{+})\backslash G(\mathbb{A}_{F^{+}})$ (over $\mathbb{C}$). For each $\wtau\in\wI_p$ we have $\iota\wtau:F\into \C$ so we may define $\xi_{\lambda_{\iota\wtau}}$ to be the continuous $\C$-representation of $G(F^+_{\iota\wtau|_{F^+}})$ of highest weight $\lambda_{\wtau}$. If $\sigma_{\lambda}$ denotes the representation of $G(F_{\infty}^{+})$ given by $\bigotimes_{\wtau\in\widetilde{I}_{p}}\xi_{\lam_{\iota\wtau}}$ then there is an $\iota$ semilinear isomorphism of $G(\A_{F^+}^{\infty})$-modules 
\begin{equation*}
\iota: S_{\lambda}(\qpbar)\xrightarrow{\sim}\mathrm{Hom}_{G(F_{\infty}^{+})}(\sigma_{\lambda}^{\vee},\mathcal{A}).
\end{equation*}
In particular $S_{\lambda}(\qpbar)$ is a semisimple admissible $G(\mathbb{A}_{F^{+}}^{\infty})$-module, and if $\pi$ is an irreducible constituent of $\mathcal{A}$ with $\pi_{\infty}\cong \sigma_{\lam}^{\vee}$ then we view $\iota^{-1}\pi^{\infty}$ as a constituent of $S_{\lam}(\Qpbar)$. 

Now come the Hecke algebras. Let $\Spl(F/F+)$ denote the set of finite places of $F^+$ that split in $F$. If $U$ is any compact open subgroup of $G(\mathbb{A}_{F^{+}}^{\infty,p})\times G(\mathcal{O}_{F^{+},p})$ we say that $U$ is unramified at a place $v\in\Spl(F/F^+)$ if $U=U^{v}\times G(\mathcal{O}_{F_{v}^{+}})$ for some $U^v\leq G(\A_{F^+}^{\infty,v})$. A compact open subgroup $U$ is thus unramified at almost all places in $\Spl(F/F^+)$. If $\Sigma$ is any finite set of finite places of $F^{+}$ containing $S_{p}$ and all $v\in\Spl(F/F^+)$ at which $U$ is ramified (in this situation we say that $\Sigma$ is \emph{good for $U^{p}$}) then we define the universal Hecke algebra $\mathbb{T}^{\Sigma,\mathrm{univ}}=\Oe[T_{w}^{(j)}\,:\,w|v\in\Spl(F/F^+)\setminus\Sigma,\,1\leq j\leq n]$ (that is, the polynomial algebra in the variables $T_{w}^{(j)}$). It acts on $S(U,W)$ by letting $T_{w}^{(j)}$ operate as the double coset 
\begin{equation*}
T_{w}^{(j)}=\iota_{w}^{-1}\left[\mathrm{GL}_{n}(\mathcal{O}_{F_{w}})\begin{pmatrix}\varpi_{w}1_{j} & 0\\
0 & 1_{n-j}
\end{pmatrix}\mathrm{GL}_{n}(\mathcal{O}_{F_{w}})\right]
\end{equation*}
where $\varpi_{w}$ is any choice of uniformizer in $F_{w}$ (the operator does not depend on this choice). 

Next we recall Galois representations associated to automorphic representations of $G$. If $\pi$ is a classical automorphic representation of $G$ contributing to $S_{\lambda}(\qpbar)$ as above, there is an associated Galois representation $r_{\iota}(\pi):G_{F}\rightarrow\mathrm{GL}_{n}(\qpbar)$ (depending on our fixed isomorphism $\iota$) that is de Rham at places diving $p$ and has Hodge type $\lam$. If $\pi$ has level prime to $p$ then $r_{\iota}(\pi)$ is crystalline at places diving $p$. Moreover, $r_{\iota}(\pi)$ satisfies local-global compatibility at all finite places (summarized in Theorem 7.2.1 of \cite{HLM}). We say that a semisimple representation $\bar{r}:G_{F}\rightarrow\mathrm{GL}_{n}(\fpbar)$ is \emph{automorphic} if there exists a weight $\lambda\in(\mathbb{Z}_{+}^{n})^{\wI_{p}}$ and an automorphic form $\pi$ contributing to $S_{\lambda}(\qpbar)$ such that $\bar{r}$ is the semisimplification of the reduction mod $p$ of $r_{\iota}(\pi)$. From now on, when we fix a Galois representation $\bar{r}:G_F\rightarrow\GL_n(\F)$ we assume its coefficient field $\F$ is large enough to contain all its eigenvalues. Saying $\bar{r}$ is automorphic is equivalent to saying that there exists a compact open subgroup $U$ as above that is ramified at each place in $\Spl(F/F^+)$ where $\bar{r}$ is ramified and a finite set of finite places $\Sigma$ good for $U^p$ such that $S(U,M_{\lambda})_{\mathfrak{m}_{\bar{r}}^{\Sigma}}\neq0$. Here $\mathfrak{m}_{\bar{r}}^{\Sigma}$ is the maximal ideal of $\T^{\Sigma,\univ}$ naturally associated to $\bar{r}$ by defining the image of $T_{w}^{(j)}$ in $\ef$ to be such that the characteristic polynomial of $\bar{r}(\mathrm{Frob}_{w})$ is equal to $\sum_{j=0}^{n} (-1)^j(\N w)^{j(j-1)/2}T_w^{(j)}X^{n-j}$. Automorphy of $\bar{r}$ implies an essential conjugate self-duality: $\bar{r}^c\cong \bar{r}^{\vee}\otimes\epsilon^{1-n}$. 

We define $S(U^{p},\ef)=\underrightarrow{\lim}\,S(U_{p}U^{p},\ef)$ as $U_{p}$ ranges over compact open subgroups of $G(\mathcal{O}_{F^{+},p})$. This has a smooth action of $G(F_{p}^{+})$, and $\mathbb{T}^{\Sigma,\mathrm{univ}}$ acts on it if $\Sigma$ is good for $U^{p}$. If $\mathfrak{m}^{\Sigma}\leq\mathbb{T}^{\Sigma,\mathrm{univ}}$ is a maximal ideal with residue field $\ef$, we have $S(U^{p},\ef)[\mathfrak{m}^{\Sigma}]\subseteq S(U^{p},\ef)_{\mathfrak{m}^{\Sigma}}\subseteq S(U^{p},\F)$ naturally. We remark that if we fix an automorphic absolutely irreducible representation $\bar{r}$ as above then it's hoped that $S(U^p,\F)[\m_{\bar{r}}^{\Sigma}]$ will match with $\prod_{v\in S_p}\bar{r}|_{G_{F_{\tv}}}$ under some kind of mod $p$ local Langlands correspondence  (at least up to some multiplicities depending on $U^p$). This is by analogy with Emerton's mod $p$ local-global compatibility theorem described in \cite{BreICM}. 

If $\sigma=\bigotimes_{\wtau\in \wI_p}\sigma_{\wtau}$ is a Serre weight for $G$ it is shown in Lemma 4.4.2 of \cite{BH} that there is an isomorphism  $\mathrm{Hom}_{G(\mathcal{O}_{F^{+},p})}(\sigma^{\vee},S(U^{p},\ef))=S(G(\mathcal{O}_{F^{+},p})U^{p},\sigma)$ which induces $\mathrm{Hom}_{G(\ofp)}(\sigma^{\vee},S(U^{p},\ef)[\mathfrak{m}^{\Sigma}])=S(G(\ofp)U^{p},\sigma)[\mathfrak{m}^{\Sigma}]$. Note that the Hecke algebra $\mathcal{H}(\sigma^{\vee}):=\bigotimes_{\wtau\in \wI_p}\mathcal{H}(\sigma_{\wtau}^{\vee})\cong\ef[T_{\sigma_{\wtau}^{\vee},1},\ldots,T_{\sigma_{\wtau}^{\vee},n-1},T_{\sigma_{\wtau}^{\vee},n}^{\pm1}]_{\wtau\in \wI_p}$ acts on this space and we define the ordinary subspace $\mathrm{Hom}_{G(\ofp)}(\sigma^{\vee},S(U^{p},\ef)[\mathfrak{m}^{\Sigma}])^{\mathrm{ord}}$ as in Section \ref{subsec:repns of GLn}.

\subsection{Ordinary automorphic forms and Galois representations}\label{subsec:aut gal repns}

Let $\bar{r}:G_{F}\rightarrow\mathrm{GL}_{n}(\ef)$ be a continuous representation. Let $\lambda\in(\mathbb{Z}_{+,\mathrm{res}}^{n})^{\wI_{p}}$ be a weight. We say that $\bar{r}$ is \emph{modular and ordinary of weight $\lambda^{\vee}$} if there exists a sufficiently small compact open subgroup $U^{p}=\prod_{v\nmid p}U_{v}$ such that $U_{v}$ is a hyperspecial maximal compact subgroup of $G(F_{v}^{+})$ for all $v$ that are inert in $F$ and is ramified at all places where $\bar{r}$ is ramified, together with a finite set of finite places $\Sigma$ good for $U^p$ such that $\mathrm{Hom}_{G(\mathcal{O}_{F^{+},p})}\left(F_{\lambda}^{\vee},S(U^{p},\ef\right)[\mathfrak{m}_{\bar{r}}^{\Sigma}])^{\mathrm{ord}}\neq0$. Equivalently we may say that $F_{\lambda}^{\vee}$ is an ordinary Serre weight of $\bar{r}$. We write $W^{\mathrm{ord}}(\bar{r})$ for the set of ordinary Serre weights of $\bar{r}$, and sometimes write $\lambda^{\vee}\in W^{\ord}(\bar{r})$ rather than $F_{\lambda}^{\vee}\in W^{\ord}(\bar{r})$.

\begin{proposition}\label{prop:ord crys lifts}
Suppose that $\bar{r}$ satisfies the following assumptions: 
\begin{itemize}
\item $\bar{r}|_{G_{F(\zeta_{p})}}$ is absolutely irreducible, where $\zeta_p$ is a primitive $p$th root of unity,
\item $\bar{r}$ is modular and ordinary of some weight,
\item $\bar{r}|_{G_{F_{\tv}}}$ is inertially generic for each $v\in S_p$
\item $p>2n+2$ and $\zeta_{p}\notin F$.
\end{itemize}
Then $\lambda^{\vee}\in W^{\mathrm{ord}}(\bar{r})$ iff $\bar{r}|G_{F_{\tilde{v}}}$ has an ordinary crystalline lift of weight $\lambda_{\wtau(\tilde{v})}$ for each $v\in S_p$. 
\end{proposition} 

\noindent The first two assumptions imply $\bar{r}_{G_{F_{\tv}}}$ is ordinary for each $v\in S_p$ by Proposition 4.4.4 of \cite{BH}, so the third assumption makes sense. Also, Proposition \ref{prop:ord crys lifts} still holds with the weaker definition of inertially generic in \cite{BH}.

\begin{proof}
The inertial genericity implies that if $\lambda^{\vee}\in W^{\ord}(\bar{r})$ then $\lambda_{\wtau}$ is generic for each $\wtau\in \wI_p$. We deduce the existence of the lifts from Proposition 4.4.4 of \cite{BH} and Lemma 3.1.5 of \cite{GG}. 

Conversely, if $\bar{r}|_{G_{F_{\tv}}}$ has an ordinary crystalline lift of weight $\lambda_{\wtau(\tv)}$ for each $v\in S_p$, then $\bar{r}|_{G_{F_{\tv}}}$ is ordinary with diagonal inertial weights 
\begin{equation*}
(-\lambda_{\wtau(\tv),n},-\lam_{\wtau(\tv),n-1},\ldots,-\lam_{\wtau(\tv),1}-(n-1))
\end{equation*}
 and the inertial genericity implies that the genericity hypothesis in Proposition 4.4.5 of \cite{BH} is satisfied. It immediately follows that $\lambda^{\vee}\in W^{\ord}(\bar{r})$.  
\end{proof}

Next we define a different “ordinary subspace” with $\Oe$-coefficients, and then relate it to the mod $p$ ordinary subspace as well as to automorphic representations. Let $\lambda\in(\mathbb{Z}_{+}^{n})^{\wI_{p}}$. Let $U^{p}\leq G(\mathbb{A}_{F^{+}}^{\infty,p})$ be any sufficiently small compact open subgroup and set $U=G(\ofp)U^{p}$. We define certain Hecke operators at $p$ that act on $S(U,M_{\lambda})$, tailored to the particular weight $\lambda$. For each $v\in S_p$ we let $T_{\lambda,\tilde{v}}^{(j)}$, $1\leq j\leq n$ denote the double coset operator
\begin{equation*}
 T_{\lambda,\tilde{v}}^{(j)}=\left(\prod_{i=1}^{j}p^{-\lambda_{\wtau(\tv),n-i+1}}\right)\cdot\iota_{\tilde{v}}^{-1}\left[\mathrm{GL}_{n}(\mathcal{O}_{F_{\tilde{v}}})\begin{pmatrix}p1_{j} & 0\\
0 & 1_{n-j}
\end{pmatrix}\mathrm{GL}_{n}(\mathcal{O}_{F_{\tilde{v}}})\right]
\end{equation*}
that \emph{a priori} acts only on $S(G(\mathcal{O}_{F^{+},p})U^{p},M_{\lambda}\otimes_{\Oe}E)$, because of the powers of $p$ out front. Note that these operators commute.

Let $\lambda\in (\Z^n_{+,\res})^{\wI_p}$. Then as explained in Section \ref{subsec:aut forms} we may view $F_{\lambda}$ as the subrepresentation of $M_{\lambda}\otimes_{\Oe}\F$ generated by a highest weight vector. Let $Q_{\lambda}$ denote the resulting quotient representation. As $U$ is sufficiently small there is an exact sequence
\begin{equation*}
0\rightarrow S(U,F_{\lambda})\rightarrow S(U,M_{\lambda})\otimes_{\Oe}\F\rightarrow S(U,Q_{\lambda})\rightarrow 0. 
\end{equation*}

\begin{proposition}\label{prop:hecke operators at p}
The action of $T_{\lambda,\tilde{v}}^{(j)}$ on $S(U,M_{\lambda}\otimes_{\Oe}E)$ stabilizes the $\Oe$-lattice $S(U,M_{\lambda})$. Moreover it stabilizes $S(U,F_{\lambda})\subseteq S(U,M_{\lambda})\otimes_{\Oe}\ef$ and its action on this subspace coincides with that of the mod $p$ Hecke operator $T_{F_{\lam_{\wtau(\tv)}}^{\vee},n-j}T_{F_{\lam_{\wtau(\tv)}}^{\vee},n}^{-1}$ under the identification $S(U,F_{\lambda})= \Hom_{G(\ofp)}(F_{\lambda}^{\vee},S(U^p,\F))$. Furthermore, the induced action of $T_{\lambda,\tilde{v}}^{(j)}$ on $S(U,Q_{\lambda})$ is $0$.
\end{proposition}

\begin{proof}
The first two claims are proved in \cite{EGH}, Proposition 4.4.2. Note that this proposition proves that the action of $T_{\lam,\tv}^{(j)}$ on $S(U,F_{\lam})$ coincides with that of $T_{F_{\lam_{\wtau(\tv)}},j}$, but it is easy to check that this corresponds to the claimed action on $\Hom_{G(\ofp)}(F_{\lam}^{\vee},S(U^p,\F))$ using the isomorphism $\cH(F_{\lam})\cong\cH(F_{\lam}^{\vee})$ in Section 2.3 of \cite{Her11}. The claim about the induced action on $S(U,Q_{\lambda})$ is part of Lemma 6.1.3 of \cite{GG}.
\end{proof}

Given the first part of this proposition there is a unique decomposition of $\Oe$-modules
\begin{equation*}
S(U,M_{\lambda})=S^{\ord}(U,M_{\lambda})\oplus S^{\mathrm{non-ord}}(U,M_{\lambda})
\end{equation*}
such that each $T_{\lambda,\tilde{v}}^{(j)}$ for $v\in S_p$ and $1\leq j\leq n$ preserves these submodules and acts invertibly on the former and their product is topologically nilpotent on the latter. Explicitly, we can write $S^{\mathrm{ord}}(U,M_{\lambda})=e\cdot S(U,M_{\lambda})$ where $e:=\lim_{n\rightarrow\infty} \left(\prod_{\tilde{v}\in\wS_p}\prod_{j=1}^n T_{\lambda,\tilde{v}}^{(j)}\right)^{n!}\in \End_{\Oe}(S(U,M_{\lambda}))$ is idempotent.

\begin{corollary}
For any $\Sigma$ that is good for $U$,
\begin{equation*}
 \mathrm{Hom}_{G(\ofp)}(F_{\lambda}^{\vee},S(U^{p},\ef)[\mathfrak{m}_{\bar{r}}^{\Sigma}])^{\mathrm{ord}}=\left(S^{\mathrm{ord}}(U,M_{\lambda})\otimes_{\Oe}\ef\right)[\mathfrak{m}_{\bar{r}}^{\Sigma}].
 \end{equation*} 
\end{corollary}

\begin{proof}
The left hand side is by definition the largest subspace of $S(U,F_{\lambda})[\m_{\bar{r}}^{\Sigma}]$  on which the mod $p$ Hecke operators act invertibly. On the other hand, $S^{\ord}(U,M_{\lambda})\otimes_{\Oe}\F$ is equal to the largest subspace of $S(U,M_{\lambda})\otimes_{\Oe}\F$ on which the operators $T_{\lambda,\tilde{v}}^{(j)}$ each act invertibly. By Proposition \ref{prop:hecke operators at p} this is equal to the largest subspace of $S(U,F_{\lambda})$ on which these operators act invertibly, but their action on $S(U,F_{\lambda})$ is equivalent to that of the mod $p$ Hecke operators and the claim follows by taking $\m_{\bar{r}}^{\Sigma}$-torsion.
\end{proof}

We define $\mathbb{T}_{\lambda}^{\Sigma,\mathrm{ord}}(U)$ to be the image of $\mathbb{T}^{\Sigma,\mathrm{univ}}$ in $S^{\mathrm{ord}}(U,M_{\lambda})$. This ring is reduced by semisimplicity of $S_{\lambda}(\qpbar)$. 

Next we recall Galois representations associated to ordinary automorphic representations of $G$.
\begin{lemma}\label{lem:ord aut repns}
Let $\pi$ be an irreducible $G(\mathbb{A}_{F^{+}}^{\infty})$-subrepresentation of $S_{\lambda}(\qpbar)$ such that $0\neq\pi^{U}\subseteq S^{\mathrm{ord}}(U,M_{\lambda})\otimes_{\Oe}\qpbar$ for some $U=G(\ofp)U^{p}$. Then $r_{\iota}(\pi)|_{G_{F_{\tilde{v}}}}$ is ordinary crystalline of weight $\lambda_{\wtau(\tv)}$ for all places $v\in S_p$. 
\end{lemma} 

\begin{proof}
This is immediate from Corollary 2.7.8 in \cite{Ger}. 
\end{proof}

This is used to construct ordinary Hecke-valued Galois representations in preparation for patching. In the next proposition, the ring $\mathbb{T}_{\lambda}^{\Sigma,\mathrm{ord}}(U)_{\mathfrak{m}_{\bar{r}}^{\Sigma}}$ is finite free over $\Oe$, so it is complete with respect to its $p$-adic topology, which is the same as its $\mathfrak{m}_{\bar{r}}^{\Sigma}$-adic topology.

\begin{proposition}\label{prop:hecke gal repns}
Suppose that $\bar{r}$ is modular and ordinary of weight $\lambda$ and absolutely irreducible, and let $U=U^{p}G(\ofp)$ be a sufficiently small compact open subgroup ramified at all places $v\in \Spl(F/F^+)$ where $\bar{r}$ is ramified, and hyperspecial maximal compact at all inert places of $F^+$ such that $S^{\mathrm{ord}}(U,M_{\lambda})_{\mathfrak{m}_{\bar{r}}^{\Sigma}}\neq0$, for some $\Sigma$ good for $U$. Then after extending coefficients there exists a continuous representation 
\begin{equation*}
r:G_{F}\rightarrow\mathrm{GL}_{n}(\mathbb{T}_{\lambda}^{\Sigma,\mathrm{ord}}(U)_{\mathfrak{m}_{\bar{r}}^{\Sigma}})
\end{equation*}
lifting $\bar{r}$ uniquely determined by the requirement that it is unramified at all places $w$ lying above $v\in \Spl(F/F^+)\setminus\Sigma$ and the characteristic polynomial of $r(\mathrm{Frob}_{w})$ is $\sum_{j=0}^{n} (-1)^j (\N w)^{j(j-1)/2}T^{(j)}_w X^{n-j}$. It has the further properties that $r^{c}\cong r^{\vee}\otimes\epsilon^{1-n}$ and that for all $v\in S_{p}$ and each minimal prime $\mathfrak{p}\leq\mathbb{T}_{\lambda}^{\Sigma,\mathrm{ord}}(U)_{\mathfrak{m}_{\bar{r}}^{\Sigma}}$, $r|_{G_{F_{\tilde{v}}}}\mod\mathfrak{p}$ is ordinary crystalline of weight $\lambda_{\wtau(\tv)}$ valued in $\Oe$.
\end{proposition}

\begin{proof}
After extending coefficients we may assume that for each minimal prime $\mathfrak{p}\leq\mathbb{T}_{\lambda}^{\Sigma,\mathrm{ord}}(U)_{\mathfrak{m}_{\bar{r}}^{\Sigma}}$ the quotient is equal to $\Oe$. For each $\mathfrak{p}$ we get a representation $G_{F}\rightarrow\mathrm{GL}_{n}(\mathbb{T}_{\lambda}^{\Sigma,\mathrm{ord}}(U)_{\mathfrak{m}_{\bar{r}}^{\Sigma}}/\mathfrak{p})$ lifting $\bar{r}$ and satisfying the analogous conditions to the proof statement by the association of Galois representations to automorphic representations of $G$ and Lemma \ref{lem:ord aut repns}. Now argue as in the proof of Proposition 3.4.4 of \cite{CHT} using reducedness of $\mathbb{T}_{\lambda}^{\Sigma,\mathrm{ord}}(U)_{\mathfrak{m}_{\bar{r}}^{\Sigma}}$ and the fact that $\bar{r}$ is absolutely irreducible.
\end{proof}

\begin{proposition}\label{prop:blggt lemma}
Suppose that $\bar{r}$ satisfies the assumptions as in Proposition \ref{prop:ord crys lifts} and moreover is only ramified at places lying above $\Spl(F/F^+)$. Let $R$ denote the set of places of $F^{+}$ lying below a place at which $\bar{r}$ is ramified and not dividing $p$. Then there exists a compact open subgroup $U=\prod_{v}U_{v}$ where $U_{v}=G(\mathcal{O}_{F_{v}}^{+})$ if $v$ splits in $F$ and does not lie in $R$, $U_{v}$ is a hyperspecial maximal compact if $v$ is inert in $F$ and does not lie in $R$ such that $S^{\mathrm{ord}}(U,M_{\lambda})_{\mathfrak{m}_{\bar{r}}^{S_{p}\cup R}}\neq0$ for all weights $\lambda\in (\Z^n_{+,\res})^{\wI_p}$ such that $\lambda^{\vee}\in W^{\mathrm{ord}}(\bar{r})$. 
\end{proposition}

\begin{proof}
By Proposition \ref{prop:ord crys lifts}, $\bar{r}|_{G_{F_{\tv}}}$ has an ordinary crystalline lift of weight $\lambda_{\wtau(\tv)}$ for each $v\in S_p$, so since ordinary crystalline representations are potentially diagonalizable, Theorem 4.4.1 of \cite{BLGGT} furnishes an automorphic representation $\pi$ of $G$ whose contribution to $S_{\lambda}(\Qp)$ is nonzero, unramified outside $S_p\cup R$, of level potentially prime to $p$ (in the terminology of \emph{loc. cit.}) and such that $r_{\iota}(\pi)$ lifts $\bar{r}$ and $r_{\iota}(\pi)|_{G_{F_{\tv}}}$ is crystalline ordinary of weight $\lambda_{\wtau(\tv)}$ for each $v\in S_p$. For the last point one has to observe that if $\rho_1\sim\rho_2$ in the notation of \cite{BLGGT} Section 1.4 and $\rho_1$ is crystalline ordinary, then $\rho_2$ is also crystalline ordinary.  Now Theorem 2.3.1 of \cite{BLGGT} implies that we can actually find $\pi$ with the same properties as above but having level prime to $p$. Finally, it follows from the proof of Lemma 2.7.6 and the easy converse to Lemma 2.7.7 of \cite{Ger} that $\pi^U$ contributes to $S^{\ord}(U,M_{\lam})\otimes_{\Oe}\Qpbar$ so the claim follows.
\end{proof}

\section{Patching}\label{sec:patching}

In this section we use Taylor-Wiles patching to obtain the main result. 

\subsection{Set-up: assumptions and Taylor-Wiles primes}\label{subsec:patching set-up}

We let $F/F^{+}$ and $G$ be as in Section \ref{subsec:aut forms} so in particular $p$ is totally split in $F$. Fix a continuous representation $\bar{r}:G_{F}\rightarrow\mathrm{GL}_{n}(\F)$. Assume from now on that $\bar{r}$ satisfies the following conditions, of which the first two (together with $\bar{r}$ absolutely irreducible) are essential, the third possibly removable, and the rest are standard technical hypotheses required for Taylor-Wiles patching:
\begin{enumerate}[label=A\arabic*]
\item $\bar{r}$ is modular and ordinary,

\item $\bar{r}|_{G_{F_{\tv}}}$ is inertially generic for each $v\in S_p$,

\item $\bar{r}$ is unramified at all finite places not dividing $p$,

\item $\bar{r}|_{G_{F(\zeta_{p})}}$ is absolutely irreducible, 

\item $p>2n+2$ and $\zeta_{p}\notin F$,

\item $(\overline{F})^{\mathrm{ker}(\mathrm{ad}(\bar{r}))}$ does not contain $F(\zeta_{p})$.
\end{enumerate}
As in Proposition \ref{prop:ord crys lifts} the other assumptions imply that $\bar{r}|_{G_{F_{\tv}}}$ is ordinary for $v\in S_p$ so the second condition makes sense.

\begin{remark} We remark that the conditions $p>2n+2$ and $\bar{r}$ absolutely irreducible imply that the image of $\bar{r}$ is adequate in the sense of \cite{Tho}. 
\end{remark}

\begin{lemma}
Under the assumptions above, we can find a place $v_{1}$ of $F^{+}$, prime to $p$, such that
\begin{itemize}
\item $v_{1}$ splits into $ww^c$ in $F$,

\item $v_{1}$ does not split completely in $F(\zeta_{p})$,

\item $\mathrm{ad}(\bar{r})(\Frob_w)=1$.
\end{itemize}
\end{lemma}

\begin{proof}
The set of places of $F$ lying above a place of $F^+$ that splits in $F$ has density 1, and by the Chebotarev density theorem and the assumption that $(\overline{F})^{\ker(\ad(\bar{r}))}$ does not contain $F(\zeta_p)$ the set of places $w$ of $F$ such that $\ad(\bar{r})(\Frob_w)=1$ and $w$ is not totally split in $F(\zeta_p)$ has positive density. Hence there is a positive density of places $w$ satisfying both conditions. Pick one and set $v_1=w|_{F^+}$. 
\end{proof}
\noindent We fix a choice of place $\tilde{v}_{1}$ lifting $v_{1}$. For the remainder of this section fix a compact open subgroup $U=\prod_{v}U_{v}$ of $G(\A_{F^+}^{\infty})$ obeying:
\begin{enumerate}[label=B\arabic*]
\item $U_{v}=G(\mathcal{O}_{F_{v}^{+}})$ for each $v$ that is split in $F$ (including $v|p$) except $v_{1}$,

\item $U_{v}$ is a hyperspecial maximal compact subgroup of $G(F_{v}^{+})$ if $v$ is inert in $F$,

\item $U_{v_{1}}$ is the preimage under $\iota_{\tilde{v}_{1}}$ of the standard Iwahori subgroup of $\mathrm{GL}_{n}(F_{\tilde{v}_{1}})$. 
\end{enumerate}
The choice of $U_{v_{1}}$ ensures that $U$ is sufficiently small because the first two bullet points satisfied by $v_1$ as well as the fact that $ v_1\nmid p$ ensure that $\N \tilde{v}_1\not\equiv 1\mod p$ so $\mathrm{Iw}_{\tilde{v}_{1}}$ has no elements of exact order $p$. By Proposition \ref{prop:blggt lemma} and the assumption that $\bar{r}$ is unramified outside $p$, we see that $S^{\mathrm{ord}}(U,M_{\lambda})_{\mathfrak{m}_{\bar{r}}^{T}}\neq0$ where $T=S_{p}\cup\{v_{1}\}$ for \emph{each} $\lambda^{\vee}\in W^{\mathrm{ord}}(\bar{r})$.

The following is the main result and its proof constitutes the rest of this section.
\begin{theorem}\label{thm:main result}
With $\bar{r}$ satisfying A1-A6 and $U$ satisfying B1-B3 above, 
\begin{equation*}
\mathrm{dim}_{\F}\left(\mathrm{Hom}_{G(\ofp)}(\sigma^{\vee},S(U^{p},\ef)[\mathfrak{m}_{\bar{r}}^{T}])^{\mathrm{ord}}\right)
\end{equation*}
 is independent of the choice of Serre weight $\sigma^{\vee}\in W^{\ord}(\bar{r})$. 
\end{theorem}

\begin{lemma}\label{lem:free over the hecke algebra}
For each $\lambda^{\vee}\in W^{\mathrm{ord}}(\bar{r})$, $S^{\mathrm{ord}}(U,M_{\lambda})_{\mathfrak{m}_{\bar{r}}^{T}}\otimes_{\Oe}\qpbar$ is free over $\mathbb{T}_{\lambda}^{T,\mathrm{ord}}(U)_{\mathfrak{m}_{\bar{r}}^{T}}\otimes_{\Oe}\qpbar$ of rank $n!$.
\end{lemma}

\begin{proof}
The former is a semisimple faithful module for the latter, which is artinian. Let $\pi$ be an automorphic representation of $G$ that contributes to $S^{\mathrm{ord}}(U,M_{\lambda})_{\mathfrak{m}_{\bar{r}}^{T}} \otimes_{\Oe}\qpbar$ and let $\Pi$ be its base change to $\mathrm{GL}_{n/F}$, which exists by Corollaire 5.3 of \cite{Lab}. Then $\Pi_{\tilde{v}_{1}}^{\mathrm{Iw}_{\tilde{v}_{1}}}\neq0$ and it follows from the Borel-Matsumoto theorem, the assumption that $\ad(\bar{r})(\Frob_{\tilde{v}_1})=1$, and local-global compatibility at $\tilde{v}_1$ that $\Pi_{\tilde{v}_{1}}$ is an (irreducible) unramified principal series representation of $\mathrm{GL}_{n}(F_{\tilde{v}_{1}})$. Its Iwahori invariants are thus $n!$-dimensional, which implies $\dim(\pi^{U})=n!$. Since all $\pi$ occur with multiplicity $1$ by Th\'eor\`eme 5.4 of \cite{Lab}, we are done.
\end{proof}

We now proceed to set up the standard data for patching. Fix a choice of weight $\lambda^{\vee}\in W^{\ord}(\bar{r})$. We start with the ``base'' data at level $U$. Let $R_{\tilde{v}_{1}}^{\square}$ denote the universal lifting ring of $\bar{r}|_{G_{F_{\tilde{v}_{1}}}}$ and for $v\in S_{p}$ write $R_{\tilde{v}}^{\triangle_{\lambda_{\tilde{v}}}}$ for the weight $\lambda_{\tilde{v}}$ ordinary crystalline deformation ring of $\bar{r}|_{G_{F_{\tilde{v}}}}$ defined in Section \ref{subsec:naive lifting rings}, both with respect to the coefficient ring $\Oe$. 

\begin{lemma}\label{lem:smooth at v1}
With $v_1$ chosen as above, $R_{\tilde{v}_1}^{\square}$ is formally smooth over $\Oe$ of dimension $n^2+1$.
\end{lemma}

\begin{proof}
The conditions on the place $v_1$ imply that $H^0(G_{F_{\tilde{v}_1}},\ad(\bar{r})(1))=0$, so the claim follows immediately from Lemma 2.4.9 of \cite{CHT}. 
\end{proof}

Set
\begin{equation*}
R_{\lambda}^{\mathrm{loc}}=\left(\widehat{\bigotimes}_{v\in S_{p}}R_{\tilde{v}}^{\triangle_{\lambda_{\tilde{v}}}}\right)\widehat{\otimes} R_{\tilde{v}_{1}}^{\square}
\end{equation*}
where the tensor products are over $\Oe$. Importantly, by the assumption of inertial genericity $\lam_{\wtau}$ is generic for each $\wtau\in \wI_p$ (see the proof of Proposition \ref{prop:ord crys lifts}). Hence $R_{\lam}^{\mathrm{loc}}$ is a power series ring over $\Oe$ in $n^{2}+\sum_{v\in S_{p}}(n^{2}+[F_{\tilde{v}}:\qp]\frac{n(n-1)}{2})=n^{2}\#T+[F^{+}:\mathbb{Q}]\frac{n(n-1)}{2}$ variables, by Proposition \ref{prop:properties of ordinary deformation rings} and Lemma \ref{lem:smooth at v1}.

Consider the following deformation problem, in the notation of \cite{CHT}:
\begin{equation*}
\mathcal{S}_{\lambda}=\left(F/F^{+},T,\tilde{T},\Oe,\bar{r},\omega^{1-n}\delta_{F/F^{+}}^{n},\{R_{\tilde{v}_{1}}^{\square}\}\cup\{R_{\tilde{v}}^{\triangle_{\lambda_{v}}}\}_{v\in S_{p}}\right).
\end{equation*}
We recall that a $T$-framed deformation of type $\mathcal{S}_{\lambda}$ to an object $A$ of $\mathcal{C}_{\Oe}$ is an equivalence class of tuples $(r,\{M_{v}\}_{v\in T})$ where $r$ is a lift $G_{F}\rightarrow\mathrm{GL}_{n}(A)$ unramified outside $T$, and the local lifts $r|_{G_{F_{\tv}}}$ at places $v\in T$ factor through the specified quotients of the universal lifting rings and $M_{v}$ is an element of $\ker(\mathrm{GL}_{n}(A)\rightarrow\mathrm{GL}_{n}(\ef))$. There is a universal $T$-framed deformation ring $R_{\cS_{\lambda}}^{\square_T}$ of type $\mathcal{S}_{\lambda}$, and a universal deformation ring $R_{\cS_{\lambda}}^{\univ}$ of type $\mathcal{S}_{\lambda}$ (which ignores the framing at places in $T$). We let $r_{\cS_{\lam}}^{\univ}:G_F\rightarrow \GL_n(R^{\univ}_{\cS_{\lam}})$ denote a choice of universal deformation, defined up to $\ker\left( \GL_n(R^{\univ}_{\cS_{\lam}})\rightarrow \GL_n(\F)\right)$-conjugacy. To any $T$-framed deformation, we can canonically associate a deformation of $\bar{r}$, as well as a collection of local lifts at places in $T$. Hence there are canonical maps of $\Oe$-algebras $R_{\lambda}^{\mathrm{loc}}\rightarrow R_{\mathcal{S}_{\lambda}}^{\square_{T}}$ and $R_{\cS_{\lambda}}^{\univ}\rightarrow R_{\cS_{\lambda}}^{\square_T}$.

\begin{lemma}
The representation $r$ of Proposition \ref{prop:hecke gal repns} is a deformation of type $\mathcal{S}_{\lambda}$. 
\end{lemma}

\begin{proof}
Because $r\mod\mathfrak{p}$ is ordinary crystalline at at all places dividing $p$ for each minimal prime $\mathfrak{p}$ and $\mathbb{T}_{\lambda}^{T,\mathrm{ord}}(U)_{\mathfrak{m}_{\bar{r}}^{T}}$ is reduced, this follows from Proposition \ref{prop:properties of ordinary deformation rings} (with $\Sigma=T$).
\end{proof}

\noindent Hence $r$ induces a map $R_{\mathcal{S}_{\lambda}}^{\mathrm{univ}}\onto\mathbb{T}_{\lambda}^{T,\mathrm{ord}}(U)_{\mathfrak{m}_{\bar{r}}^{T}}$, which is surjective because its image contains all the $T_{w}^{(j)}$ by the defining property of $r$ in Proposition \ref{prop:hecke gal repns}. 

Now we pick Taylor-Wiles primes $(Q_{N})_{N\geq1}$ using a minor variant of the method of \cite{Tho} to account for the fact that we want ordinary deformation rings at $p$.

\begin{proposition}\label{prop:Taylor-Wiles primes}
There is an integer $q\geq[F^{+}:\mathbb{Q}]\frac{n(n-1)}{2}$ such that for each integer $N\geq1$ there exists a tuple $(Q_{N},\tilde{Q}_{N},\{\bar{\psi}_{\tv}\}_{v\in Q_{N}})$ where 
\begin{itemize}
\item $Q_{N}$ is a set of cardinality $q$ consisting of places of $F^{+}$ that split in $F$, disjoint from $T$,

\item  $\mathbb{N}v\equiv1\mod p^{N}$ for every $v\in Q_{N}$,

\item $\tilde{Q}_{N}$ is a set of places of $F$ consisting of one $\tilde{v}$ dividing $v$ for each $v\in Q_{N}$,

\item for each $v\in Q_{N}$, $\bar{r}|_{G_{F_{\tilde{v}}}}=\bar{s}_{\tilde{v}}\oplus\bar{\psi}_{\tilde{v}}$ where $\bar{r}(\mathrm{Frob}_{\tilde{v}})$ acts via a scalar $\bar{\alpha}_{\tilde{v}}\in \F^{\times}$ on $\bar{\psi}_{\tilde{v}}$. 
\end{itemize}
For each $v\in Q_{N}$ denote by $R_{\tilde{v}}^{\bar{\psi}_{\tilde{v}}}$ the quotient of $R_{\tilde{v}}^{\square}$ corresponding to lifts $G_{F_{\tilde{v}}}\rightarrow\mathrm{GL}_{n}(A)$ of $\bar{r}|_{G_{F_{\tilde{v}}}}$ that are $\ker(\mathrm{GL}_{n}(A)\rightarrow\mathrm{GL}_{n}(\ef))$-conjugate to a lift of the form $s\oplus\psi$ where $s$ is an unramified lift of $\bar{s}_{\tilde{v}}$ and $\psi$ is a lift of $\bar{\psi}_{\tilde{v}}$ for which the image of inertia is contained in the scalar matrices (but which may be ramified). Let $\mathcal{S}_{\lambda,Q_{N}}$ denote the deformation problem
\begin{equation*}
\left(F/F^{+},T\cup Q_{N},\tilde{T}\cup\tilde{Q}_{N},\Oe,\bar{r},\epsilon^{1-n}\delta_{F/F^{+}}^{n},\{R_{\tilde{v}_{1}}^{\square}\}\cup\{R_{\tilde{v}}^{\triangle_{\lambda}}\}_{v\in S_{p}}\cup\{R_{\tilde{v}}^{\bar{\psi}_{\tilde{v}}}\}_{v\in Q_{N}}\right)
\end{equation*}
and let $R_{\mathcal{S}_{\lambda},Q_{N}}^{\square_{T}}$ denote the corresponding $T$-framed (not $T\cup Q_{N}$-framed) universal deformation ring and $R_{\mathcal{S}_{\lambda,Q_{N}}}^{\mathrm{univ}}$ the universal deformation ring of type $\mathcal{S}_{\lambda,Q_{N}}$. Finally we have the canonical map $R_{\lambda}^{\mathrm{loc}}\rightarrow R_{\mathcal{S}_{\lambda,Q_N}}^{\square_{T}}$ and we can pick the $Q_{N}$ so that furthermore
\begin{itemize}
\item $R_{\mathcal{S}_{\lambda,Q_{N}}}^{\square_{T}}$ can be topologically generated over $R_{\lambda}^{\mathrm{loc}}$ by $q':=q-[F^{+}:\mathbb{Q}]\frac{n(n-1)}{2}$ elements.
\end{itemize}
\end{proposition}

\begin{proof}
If we replace the deformation problem $\mathcal{S}_{\lambda,Q_{N}}$ with
\begin{equation*}
 \mathcal{S}_{\lambda,Q_{N}}'=\left(F/F^{+},T\cup Q_{N},\tilde{T}\cup\tilde{Q}_{N},\Oe,\bar{r},\epsilon^{1-n}\delta_{F/F^{+}}^{n},\{R_{\tilde{v}_{1}}^{\square}\}\cup\{R_{\tilde{v}}^{\square}\}_{v\in S_{p}}\cup\{R_{\tilde{v}}^{\bar{\psi}_{\tilde{v}}}\}_{v\in Q_{N}}\right) 
 \end{equation*}
 and set
 \begin{equation*}
 R_{\lambda}^{\mathrm{loc}'}=\left(\widehat{\bigotimes}_{v\in S_{p}}R_{\tilde{v}}^{\square}\right)\widehat{\otimes}R_{\tilde{v}_{1}}^{\square}
 \end{equation*}
 then it follows from Proposition 4.4 of \cite{Tho} that we can find tuples $(Q_{N},\tilde{Q}_{N},\{\bar{\psi}_{v}\}_{v\in Q_{N}})$ obeying the first four conditions and such that $R_{\mathcal{S}_{\lambda,Q_{N}}'}^{\square_{T}}$ can be topologically generated over $R_{\lambda}^{\mathrm{loc}'}$ by $q'$ elements. This uses the adequacy of the image of $\bar{r}$. Now from the commutative diagram
 \begin{equation*}
 \xymatrix{ R_{\lambda}^{\mathrm{loc}'}\ar[r] \ar@{->>}[d] & R_{\cS'_{\lambda,Q_N}}^{\square_T}\ar@{->>}[d] \\ R_{\lambda}^{\mathrm{loc}}\ar[r] & R^{\square_T}_{\cS_{\lambda,Q_N}} } 
 \end{equation*}
 we immediately see that the fifth condition is also satisfied.
\end{proof}

Write $d_{N}^{\tilde{v}}=\dim(\bar{\psi}_{\tilde{v}})$ for any place $v\in Q_{N}$. Let $U_{i}(Q_{N})=\prod_{v}U_{i}(Q_{N})_{v}$ for $i=0,1$ where $U_{i}(Q_{N})_{v}=U_{v}$ if $v\notin Q_{N}$, but for $v\in Q_{N}$ we have $U_{0}(Q_{N})_{v}=\iota_{\tilde{v}}^{-1}\mathfrak{p}_{N}^{\tilde{v}}$ and $U_{1}(Q_{N})_{v}=\iota_{\tilde{v}}^{-1}\mathfrak{p}_{N,1}^{\tilde{v}}$ where $\mathfrak{p}_{N}^{\tilde{v}}$ is the standard parahoric subgroup of $\mathrm{GL}_{n}(F_{\tilde{v}})$ corresponding to the partition $n=(n-d_{N}^{\tilde{v}})+d_{N}^{\tilde{v}}$ and $\mathfrak{p}_{N,1}^{\tilde{v}}$ is the kernel of the map
\begin{equation*} \mathfrak{p}_{N}^{\tilde{v}}\twoheadrightarrow\mathrm{GL}_{d_{N}^{\tilde{v}}}(k_{\tilde{v}})\twoheadrightarrow k_{\tilde{v}}^{\times}\twoheadrightarrow k_{\tilde{v}}^{\times}(p)
\end{equation*}
where the second arrow is the determinant and the third arrow is the projection to the maximal $p$-power order quotient $k_{\tilde{v}}^{\times}(p)$ of $k_{\tilde{v}}^{\times}$. We define $\Delta_{Q_{N}}:=U_{0}(Q_{N})/U_{1}(Q_{N})=\prod_{v\in Q_{N}}k_{\tilde{v}}^{\times}(p)$, and let $\a_{Q_N}$ denote the augmentation ideal of $\Oe[\Delta_{Q_N}]$ (that is, the ideal generated by all elements $\delta-1$ with $\delta\in\Delta_{Q_N}$). Note that because $\Delta_{Q_N}$ is a $p$-group, $\Oe[\Delta_{Q_N}]$ is a local ring with maximal ideal $(\varpi_E)+\a_{Q_N}$. 

There are natural maps
\begin{equation*}
\mathbb{T}_{\lambda}^{T\cup Q_{N},\mathrm{ord}}(U_{1}(Q_{N}))\twoheadrightarrow\mathbb{T}_{\lambda}^{T\cup Q_{N},\mathrm{ord}}(U_{0}(Q_{N}))\twoheadrightarrow\mathbb{T}_{\lambda}^{T\cup Q_{N},\mathrm{ord}}(U)\hookrightarrow\mathbb{T}_{\lambda}^{T,\mathrm{ord}}(U)
\end{equation*}
via which $\mathfrak{m}_{\bar{r}}^{T}\leq\mathbb{T}_{\lambda}^{T,\mathrm{ord}}(U)$ determines maximal ideals of each of the first three algebras. We denote these maximal ideals by $\mathfrak{m}_{Q_{N}}$ in the first two cases and $\mathfrak{m}$ in the second two (thus changing notation). After localizing at $\m$ the last map is an isomorphism by the argument of Corollary 3.4.5 of \cite{CHT}. Observe that $\Delta_{Q_{N}}$ has a natural action on $S^{\mathrm{ord}}(U_{1}(Q_{N}),M_{\lambda})_{\mathfrak{m}_{Q_{N}}}$. Also note that $S^{\mathrm{ord}}(U,M_{\lambda})_{\mathfrak{m}}$ is naturally a submodule of $S^{\mathrm{ord}}(U_{i}(Q_{N}),M_{\lambda})_{\mathfrak{m}_{Q_{N}}}$.

We proceed to define certain data following Sections 5 and 6 of \cite{Tho} at level $U_i(Q_N)$ for a fixed $N\geq 1$. For each $v\in Q_{N}$ choose a geometric Frobenius lift $\phi_{\tilde{v}}$ in $G_{F_{\tilde{v}}}$ and let $\varpi_{\tilde{v}}$ be a uniformizer corresponding to this lift under the local reciprocity map. For each $\alpha\in\mathcal{O}_{F_{\tilde{v}}}^{\times}$ write $V_{\tv,\alpha}$ for the Hecke operator $\iota_{\tilde{v}}^{-1}\left[\mathfrak{p}_{N,1}^{\tilde{v}} \begin{pmatrix}1_{n-d_{N}^{\tilde{v}}} & 0\\ 0 & A_{\alpha} \end{pmatrix}\mathfrak{p}_{N,1}^{\tilde{v}}\right]$ where $A_{\alpha}=\mathrm{diag}(\alpha,1,\ldots,1)$. As in Proposition 5.9 of \cite{Tho} we define certain projection\footnote{These operators are not projections in the sense of being idempotent, but their images are direct summands. The terminology ``projection'' seems to be standard.} operators $\mathrm{pr}_{\varpi_{\tilde{v}}}\in\mathrm{End}_{\Oe}(S^{\mathrm{ord}}(U_{i}(Q_{N}),M_{\lambda})_{\mathfrak{m}_{Q_{N}}})$ for each $v\in Q_{N}$ compatible between $i=0$ and $i=1$. They commute with one another, and with the action of $\Oe[\Delta_{Q_{N}}]$ on $S^{\mathrm{ord}}(U_{1}(Q_{N}),M_{\lambda})_{\mathfrak{m}_{Q_{N}}}$. Thus we define from now on $\mathrm{pr}_{N}=\prod_{v\in Q_{N}}\mathrm{pr}_{\varpi_{\tilde{v}}}$ unambiguously and $\pr_N\left(S^{\mathrm{ord}}(U_{1}(Q_{N}),M_{\lambda})_{\mathfrak{m}_{Q_{N}}}\right)$ has an action of $\Oe[\Delta_{Q_{N}}]$. Let $\mathbb{T}_{1,Q_{N}}$ denote the image of $\mathbb{T}_{\lambda}^{T\cup Q_{N},\mathrm{ord}}(U_{1}(Q_{N}))_{\mathfrak{m}_{Q_{N}}}$ in \begin{equation*}
\mathrm{End}_{\Oe}\left(\mathrm{pr}_{N}\left(S^{\mathrm{ord}}(U_{1}(Q_{N}),M_{\lambda})_{\mathfrak{m}_{Q_{N}}}\right)\right).
\end{equation*}
 In the next proposition we make use of the representation 
\begin{equation*}
r_{Q_{N}}:G_{F}\rightarrow\mathrm{GL}_{n}(\mathbb{T}_{\lambda}^{T\cup Q_{N},\mathrm{ord}}(U_{1}(Q_{N}))_{\mathfrak{m}_{Q_{N}}})
\end{equation*}
 defined in Proposition \ref{prop:hecke gal repns} (with $\Sigma=T\cup Q_N$).

\begin{proposition}\label{prop:properties of projections} The following are true:
\begin{itemize}
\item the map $\mathrm{pr}_{N}$ restricts to an isomorphism
\begin{equation*}
 S^{\mathrm{ord}}(U,M_{\lambda})_{\mathfrak{m}}\xrightarrow{\sim} \mathrm{pr}_{N}\left(S^{\mathrm{ord}}(U_{0}(Q_{N}),M_{\lambda})_{\mathfrak{m}_{Q_{N}}}\right),
 \end{equation*}

\item $\mathrm{pr}_{N}\left(S^{\mathrm{ord}}(U_{1}(Q_{N}),M_{\lambda})_{\mathfrak{m}_{Q_{N}}}\right)$ is a finite free $\Oe[\Delta_{Q_{N}}]$-module and its sub-$\Oe$-module of $\Delta_{Q_{N}}$-invariants is $\mathrm{pr}_{N}\left(S^{\mathrm{ord}}(U_{0}(Q_{N}),M_{\lambda})_{\mathfrak{m}_{Q_{N}}}\right)$. 

\item for each $v\in Q_{N}$ there is a character with open kernel $V_{\tilde{v}}:\mathcal{O}_{F_{\tilde{v}}}^{\times}\rightarrow\mathbb{T}_{1,Q_{N}}^{\times}$ such that
\begin{enumerate}
\item for each $\alpha\in\mathcal{O}_{F_{\tilde{v}}}^{\times}$, $V_{\tv,\alpha}=V_{\tilde{v}}(\alpha)$ on $\mathrm{pr}_{N}\left(S^{\mathrm{ord}}(U_{1}(Q_{N}),M_{\lambda})_{\mathfrak{m}_{Q_{N}}})\right)$, and

\item  $r_{Q_{N}}\otimes_{\mathbb{T}_{\lambda}^{T\cup Q_{N},\mathrm{ord}}(U_{1}(Q_{N}))_{\mathfrak{m}_{Q_{N}}}}\mathbb{T}_{1,Q_{N}}=s\oplus\psi$ where $s$ is an unramified lift of $\bar{s}_{\tilde{v}}$ and $\psi$ is a lift of $\bar{\psi}_{\tilde{v}}$ on which $I_{F_{\tilde{v}}}$ acts by the scalar character $V_{\tilde{v}}\circ\mathrm{Art}_{F_{\tilde{v}}}^{-1}$. 
\end{enumerate}
\end{itemize}
\end{proposition}

\begin{proof}
The projection operators are defined as in Proposition 5.9 of \cite{Tho} (their definition and the proofs of the theorems take place locally at places in $Q_{N}$, so the constructions of \cite{Tho} all go through with spaces of ordinary forms). Then the first bullet point is a consequence of Proposition 5.9 \emph{op. cit.}. The third bullet point is a consequence of Proposition 5.12 \emph{op. cit.}. Since $U_{0}(Q_{N})$ is sufficiently small and $\Delta_{Q_N}$ has $p$-power order, Lemma 6.4 of \cite{Tho} says that $S(U_1(Q_N),M_{\lambda})$ is a free $\Oe[\Delta_{Q_N}]$-module, and $S(U_{1}(Q_{N}),M_{\lambda})^{\Delta_{Q_{N}}}=S(U_{0}(Q_{N}),M_{\lambda})$. Since $\pr_N\left(S^{\ord}(U_1(Q_N),M_{\lambda})_{\m_{Q_N}}\right)$ is a direct summand of $S(U_1(Q_N),M_{\lambda})$ as an $\Oe[\Delta_{Q_N}]$-module, we deduce the statement about $\Delta_{Q_N}$-invariants in the second bullet point. Finally since $\Oe[\Delta_{Q_N}]$ is local the freeness statement also follows.
\end{proof}

In particular the homomorphism $r_{Q_{N}}\otimes\mathbb{T}_{1,Q_{N}}$ lifts $\bar{r}$ and is of type $\mathcal{S}_{\lambda,Q_{N}}$ so it induces a map $R_{\mathcal{S}_{\lambda,Q_{N}}}^{\mathrm{univ}}\twoheadrightarrow\mathbb{T}_{1,Q_{N}}$ which is surjective by the same argument as before. Also, thinking of $\Delta_{Q_{N}}$ as the maximal $p$-power quotient of the image of $\prod_{v\in Q_{N}}I_{F_{\tilde{v}}}^{\ab}$ inside $\prod_{v\in Q_N} G_{F_{\tilde{v}}}^{\ab}$, we obtain a map $\Delta_{Q_{N}}\rightarrow(R_{\mathcal{S}_{\lambda,Q_{N}}}^{\mathrm{univ}})^{\times}$ by decomposing $r_{\mathcal{S}_{\lambda,Q_{N}}}^{\mathrm{univ}}|_{G_{F_{\tilde{v}}}}=s\oplus\psi$ in some basis and noting that $\psi:I_{F_{\tilde{v}}}\rightarrow(R_{\mathcal{S}_{\lambda,Q_{N}}}^{\mathrm{univ}})^{\times}$ is trivial modulo the maximal ideal of $R_{\mathcal{S}_{\lambda,Q_{N}}}^{\mathrm{univ}}$ and therefore factors through $\Delta_{Q_N}$. This doesn't depend on the choice of $r^{\univ}_{\cS_{\lam,Q_N}}$. Hence we have homomorphisms
\begin{equation*}
 \Oe[\Delta_{Q_{N}}]\rightarrow R_{\mathcal{S}_{\lambda,Q_{N}}}^{\mathrm{univ}}\rightarrow R_{\mathcal{S}_{\lambda,Q_{N}}}^{\square_{T}}. 
 \end{equation*}
 
 \begin{lemma}
There are natural surjections $R_{S_{\lambda,Q_{N}}}^{\mathrm{univ}}\twoheadrightarrow R_{\mathcal{S}_{\lambda}}^{\mathrm{univ}}$ and $R_{\mathcal{S}_{\lambda,Q_{N}}}^{\square_{T}}\twoheadrightarrow R_{\mathcal{S}_{\lambda}}^{\square_{T}}$ which induce isomorphisms $R_{\mathcal{S}_{\lambda,Q_{N}}}^{\mathrm{univ}}/\mathfrak{a}_{Q_{N}}\cong R_{\mathcal{S}_{\lambda}}^{\mathrm{univ}}$ and $R_{\mathcal{S}_{\lambda,Q_{N}}}^{\square_{T}}/\mathfrak{a}_{Q_{N}}\cong R_{\mathcal{S_{\lambda}}}^{\square_{T}}$ respectively. 
 \end{lemma}
 
 \begin{proof}
 One checks that a deformation of type $\cS_{\lambda}$ is automatically of type $\cS_{\lam,Q_N}$, which gives the desired surjection, and that a deformation of type $\cS_{\lam,Q_N}$ is of type $\cS_{\lam}$ iff it is unramified at places in $Q_N$ iff the action of $\Delta_{Q_N}$ is trivial. The exact same is true keeping track of the framing at places in $T$.
 \end{proof}
 
\noindent Finally, we note that by the third bullet point in Proposition \ref{prop:properties of projections} the $\Oe[\Delta_{Q_{N}}]$-module structure on $\pr_N\left(S^{\mathrm{ord}}(U_{1}(Q_{N}),M_{\lambda})_{\mathfrak{m}_{Q_{N}}}\right)$ coming from $\Oe[\Delta_{Q_{N}}]\rightarrow R_{\mathcal{S}_{\lambda,Q_{N}}}^{\mathrm{univ}}\twoheadrightarrow\mathbb{T}_{1,Q_{N}}$ agrees with its natural $\Oe[\Delta_{Q_{N}}]$-module structure.

\subsection{Proof of the main result}\label{subsec:proof of main result}

In this section we patch dual spaces of ordinary automorphic forms using the set-up of the previous section, in order to prove Theorem \ref{thm:main result}. 

We need to introduce a bit more notation. Define $M=S^{\mathrm{ord}}(U,M_{\lambda})_{\mathfrak{m}}^{\d}$ and $M_{1,Q_{N}}=\mathrm{pr}_{N}\left(S^{\mathrm{ord}}(U_{1}(Q_{N}),M_{\lambda})_{\mathfrak{m}_{Q_{N}}}\right)^{\d}$, or equivalently $\mathrm{pr}_{N}^{\d}\left(S^{\mathrm{ord}}(U,M_{\lambda})_{\mathfrak{m}_{Q_{N}}}^{\d}\right)$. Because $\Delta_{Q_{N}}$ is an abelian $p$-group, we have $\Oe[\Delta_{Q_{N}}]^{\d}\cong\Oe[\Delta_{Q_{N}}]$ as a $\Oe[\Delta_{Q_{N}}]$-module, so Proposition \ref{prop:properties of projections} implies that $M_{1,Q_{N}}$ is finite free over $\Oe[\Delta_{Q_{N}}]$ and $M_{1,Q_{N}}/\mathfrak{a}_{Q_{N}}\cong M$. Furthermore since $\mathbb{T}_{\lambda}^{T,\mathrm{ord}}(U)_{\mathfrak{m}}$ and $\mathbb{T}_{1,Q_{N}}$ are commutative we may identify them with subalgebras of $\mathrm{End}_{\Oe}(M)$ and $\mathrm{End}_{\Oe}(M_{1,Q_{N}})$ respectively, and this makes $M$ and $M_{1,Q_{N}}$ into an $R_{\mathcal{S}_{\lambda}}^{\mathrm{univ}}$- and an $R_{\mathcal{S}_{\lambda,Q_{N}}}^{\mathrm{univ}}$-module respectively such that the isomorphism $M_{1,Q_{N}}/\mathfrak{a}_{Q_{N}}\cong M$ above is equivariant for these structures.

\begin{proof}[Proof of Theorem \ref{thm:main result}]
Make a choice of universal deformation $r_{\mathcal{S}_{\lambda}}^{\mathrm{univ}}:G_{F}\rightarrow\mathrm{GL}_{n}(R_{\mathcal{S}_{\lambda}}^{\mathrm{univ}})$. This determines an isomorphism $R_{\mathcal{S}_{\lambda}}^{\square_{T}}\cong R_{\mathcal{S}_{\lambda}}^{\mathrm{univ}}\hat{\otimes}\mathfrak{J}$ where we define $\mathfrak{J}:=\Oe[[X_{v,i,j}\,:\,v\in T,\ 1\leq i,j\leq n]]$. Similarly, for each $N\geq 1$ make a choice of universal deformation $r_{\mathcal{S}_{\lambda,Q_{N}}}^{\mathrm{univ}}:G_{F}\rightarrow\mathrm{GL}_{n}(R_{\mathcal{S}_{\lambda,Q_{N}}}^{\mathrm{univ}})$ which reduces to $r_{\mathcal{S}_{\lambda}}^{\mathrm{univ}}$ modulo $\mathfrak{a}_{Q_{N}}$. These determine isomorphisms $R_{\mathcal{S}_{\lambda,Q_{N}}}^{\square_{T}}\cong R_{\mathcal{S}_{\lambda,Q_{N}}}^{\mathrm{univ}}\hat{\otimes}\mathfrak{J}$ compatible with the previous one modulo $\mathfrak{a}_{Q_{N}}$. 

Define $M^{\square}=M\otimes_{R_{\mathcal{S}_{\lambda}}^{\mathrm{univ}}} R_{\mathcal{S}_{\lambda}}^{\square_{T}}$ and $M_{1,Q_{N}}^{\square}=M_{1,Q_{N}}\otimes_{R_{\mathcal{S}_{\lambda,Q_{N}}} ^{\mathrm{univ}}}R_{\mathcal{S}_{\lambda,Q_{N}}}^{\square_{T}}$. Then $M_{1,Q_{N}}^{\square}/\mathfrak{a}_{Q_{N}}\cong M^{\square}$. Set $S_{\infty}=\Oe[[\mathbb{Z}_{p}^{q}]]\hat{\otimes}\mathfrak{J}\cong\Oe[[Y_{1},\ldots,Y_{q}]]\hat{\otimes}\mathfrak{J}$, $R_{\infty}=R^{\mathrm{loc}}[[x_{1},\ldots,x_{q'}]]$ and $\mathfrak{a}_{\infty}=\ker(S_{\infty}\rightarrow\Oe)$ where the last map sends each $Y_{k}$ and $X_{v,i,j}$ to $0$. For each $N\geq1$, observe that we can choose an isomorphism $\Delta_{Q_{N}}/(\Delta_{Q_{N}})^{p^{N}}\cong(\mathbb{Z}/p^{N})^{q}$. Let $\mathfrak{c}_{N}$ be the kernel of the resulting surjection $\Oe[\Delta_{Q_{N}}]\twoheadrightarrow\Oe[(\mathbb{Z}/p^{N})^{q}]=:S_{N}$. Then we have a map $S_{N}\rightarrow R_{\mathcal{S}_{\lambda,Q_{N}}}^{\mathrm{univ}}/\mathfrak{c}_{N}$ and $M_{1,Q_{N}}/\mathfrak{c}_{N}$ is finite free over $S_{N}$ and moreover there is an isomorphism $M_{1,Q_{N}}/\mathfrak{a}_{S_{N}}\cong M$, where $\mathfrak{a}_{S_{N}}$ denotes the augmentation ideal of $S_{N}$ (whose inverse image in $\Oe[\Delta_{Q_{N}}]$ is $\mathfrak{a}_{Q_{N}}$). By the last of the defining properties of the sets $Q_{N}$ in Prop \ref{prop:Taylor-Wiles primes} we may choose surjections of $R_{\lambda}^{\mathrm{loc}}$-algebras $R_{\infty}\twoheadrightarrow R_{\mathcal{S}_{\lambda}}^{\square_{T}}$ and $R_{\infty}\twoheadrightarrow R_{\mathcal{S}_{\lambda,Q_{N}}}^{\square_{T}}$ for each $N\geq1$. We can ensure that these commute with the isomorphisms $(R_{\mathcal{S}_{\lambda,Q_{N}}}^{\square_{T}}/\mathfrak{c}_{N})/\mathfrak{a}_{S_{N}}\cong R_{\mathcal{S}_{\lambda}}^{\square_{T}}$. 

At this point we can apply the patching Lemma \ref{lem:patching} below taking $q, M, S_{N}, R_{\infty}$ as defined above and setting $R=R_{\mathcal{S}_{\lambda}}^{\mathrm{univ}}$, $R_{N}=R_{\mathcal{S}_{\lambda,Q_{N}}}^{\mathrm{univ}}/\mathfrak{c}_{N}$, $M_{N}=M_{1,Q_{N}}/\mathfrak{c}_{N}$, and $t=n^{2}\#T$ in order to get an $R_{\infty}$-module $M_{\infty}$ together with a map $S_{\infty}\rightarrow R_{\infty}$ and a map $\phi_{\infty}:R_{\infty}\rightarrow\mathbb{T}_{\lambda}^{T,\mathrm{ord}}(U)_{\mathfrak{m}}$ such that $\ker(\phi_{\infty})\supseteq\mathfrak{a}_{\infty}R_{\infty}$ and $M_{\infty}$ is finite free over $S_{\infty}$ and there is a $\phi_{\infty}$-equivariant isomorphism $M_{\infty}/\mathfrak{a}_{\infty}\cong M$. The crucial point of the patched module $M_{\infty}$ is that we now have
\begin{equation*}
1+n^{2}\#T+q=\dim(S_{\infty})=\mathrm{depth}_{S_{\infty}}(M_{\infty})\leq\mathrm{depth}_{R_{\infty}}(M_{\infty})
\end{equation*}
where the second equality is because $M_{\infty}$ is free over $S_{\infty}$ and the inequality is because the $S_{\infty}$-action on $M_{\infty}$ factors through $R_{\infty}$. On the other hand by construction $\dim(R_{\infty})=\dim(R_{\lambda}^{\mathrm{loc}})+q'=1+n^{2}\#T+q$ and it follows that $M_{\infty}$ is a maximal Cohen-Macaulay $R_{\infty}$-module. Because $R_{\infty}$ is a regular local ring, the Auslander-Buchsbaum-Serre theorem now implies that $M_{\infty}$ is free over $R_{\infty}$. We deduce from $M_{\infty}/\mathfrak{a}_{\infty}\cong M$ that $\phi_{\infty}:R_{\infty}/\mathfrak{a}_{\infty}\cong\mathbb{T}_{\lambda}^{T,\mathrm{ord}}(U)_{\mathfrak{m}}$. Thus $M$ is free over $\mathbb{T}_{\lambda}^{T,\mathrm{ord}}(U)_{\mathfrak{m}}$. But $M\otimes_{\Oe}\qpbar=(S^{\mathrm{ord}}(U,M_{\lambda})_{\mathfrak{m}}\otimes_{\Oe}\qpbar)^{\vee}$ and this space is free over $\T_{\lam}^{T,\ord}(U)_{\m}\otimes_{\Oe}\Qpbar$ of rank $n!$ by Lemma \ref{lem:free over the hecke algebra} and the fact that $\mathbb{T}_{\lambda}^{T,\mathrm{ord}}(U)_{\mathfrak{m}}\otimes_{\Oe}\qpbar$, being a product of copies of $\qpbar$, is a Gorenstein $\qpbar$-algebra. We deduce that $M/\mathfrak{m}=((S^{\mathrm{ord}}(U,M_{\lambda})\otimes_{\Oe}\ef)[\mathfrak{m}])^{\vee}$ is of dimension $n!$ over $\ef$, which implies that $\dim_{\F}\Hom_{G(\ofp)}(F_{\lam}^{\vee},S(U^p,\F)[\mathfrak{m}])^{\ord}=n!$. As this is independent of $\lambda$, this finishes the proof of Theorem \ref{thm:main result}.
\end{proof}

\begin{lemma}\label{lem:patching}
Fix positive integers $q$ and $t$. Let $R$ be an object of $\mathcal{C}_{\Oe}$ and $M$ an $R$-module that is finitely generated over $\Oe$. Suppose for each $N\geq1$ there exists an object $R_{N}$ of $\mathcal{C}_{\Oe}$ and an $R_{N}$-module $M_{N}$ finitely generated over $\Oe$ together with an $\Oe$-algebra homomorphism $S_{N}:=\Oe[(\mathbb{Z}/p^{N})^{q}]\rightarrow R_{N}$ making $M_{N}$ free over $S_{N}$ and that there exist compatible isomorphisms $R_{N}/\mathfrak{a}_{S_{N}}\cong R$ and $M_{N}/\mathfrak{a}_{S_{N}}\cong M$, where $\mathfrak{a}_{S_{N}}$ is the augmentation ideal of $S_{N}$.

 Define $\mathfrak{J}=\Oe[[x_{1},\ldots,x_{t}]]$ and $S_{\infty}=\Oe[[\mathbb{Z}_{p}^{q}]]\hat{\otimes}\mathfrak{J}=\Oe[[x_{1},\ldots,x_{t},y_{1},\ldots,y_{q}]]$. Let $\mathfrak{a}_{\infty}$ be the kernel of the map $S_{\infty}\rightarrow\Oe$ taking each $x_{i},y_{j}$ to $0$. Suppose that there is an object $R_{\infty}$ in $\mathcal{C}_{\Oe}$ and surjections $R_{\infty}\twoheadrightarrow R\hat{\otimes}\mathfrak{J}$ and $R_{\infty}\twoheadrightarrow R_{N}\hat{\otimes}\mathfrak{J}$ for each $N$ commuting with the isomorphism $R_{N}\hat{\otimes}\mathfrak{J}/\mathfrak{a}_{S_{N}}\cong R\hat{\otimes}\mathfrak{J}$.
 
Then there exists an $R_{\infty}$-module $M_{\infty}$, an $\Oe$-algebra homomorphism $S_{\infty}\rightarrow R_{\infty}$, and an $\Oe$-algebra homomorphism $\phi_{\infty}:R_{\infty}\rightarrow R$ sending $\mathfrak{a}_{\infty}R_{\infty}$ into $\mathrm{Ann}_{R}(M)$ such that
\begin{enumerate}
\item $M_{\infty}$ is finite free over $S_{\infty}$, and
\item there is an $\Oe$-module isomorphism $M_{\infty}/\a_{\infty}\cong M$ compatible with the $R_{\infty}$-action and the $R$-action via $\phi_{\infty}$. 
\end{enumerate}
\end{lemma}

\begin{proof}
For each $N\geq1$ let $p_{N}:S_{\infty}\rightarrow S_{N}\hat{\otimes}\mathfrak{J}$ denote the natural map. Regard $R_{N}\hat{\otimes}\mathfrak{J}$ as an $S_{\infty}$-algebra via $p_{N}$. Then $R_{N}\hat{\otimes}\mathfrak{J}/\mathfrak{a}_{\infty}\cong R$ and $M_{N}\hat{\otimes}\mathfrak{J}/\mathfrak{a}_{\infty}\cong M$ compatibly. This also gives $R$ an $S_{\infty}$-algebra structure. Choose any sequence of open ideals $(\mathfrak{b}_{N})_{N\geq1}$ of $S_{\infty}$ such that $\mathfrak{b}_{N}\supseteq\mathrm{ker}(p_{N})$, $\mathfrak{b}_{N}\supseteq\mathfrak{b}_{N+1}$ and $\cap_{N}\mathfrak{b}_{N}=0$. For example we could take $\mathfrak{b}_{N}=(\varpi_{E}^{N},x_{i}^{N},(1+y_{j})^{p^{N}}-1)$. Then $S_{\infty}/\mathfrak{b}_{N}$ is a (literally) finite ring. Choose a sequence of open ideal $(\delta_{N})_{N\geq1}$ of $R$ such that $\mathfrak{b}_{N}R\subseteq\delta_{N}\subseteq\mathfrak{b}_{N}R+\mathrm{Ann}_{R}(H), \delta_{N}\supseteq\delta_{N+1}$ and $\cap_{N}\delta_{N}=0$. With the specific choice of $\mathfrak{b}_{N}$ before we could for example take $\delta_{N}=\varpi_{E}^{N}R$.

Define a patching datum of level $N\geq1$ to be a tuple $(\phi,\mathfrak{M},\psi)$ where
\begin{itemize}
\item $\phi:R_{\infty}\twoheadrightarrow R/\delta_{N}$ is a surjective $\Oe$-algebra homomorphism,
\item  $\mathfrak{M}$ is a module for $R_{\infty}$ and $S_{\infty}$, the two actions commuting, that is finite free over $S_{\infty}/\mathfrak{b}_{N}$, and
\item $\psi:\mathfrak{M}/\mathfrak{a}_{\infty}\rightarrow M/\mathfrak{b}_{N}$ is an $\Oe$-module isomorphism such that $\psi(r_{\infty}m)=\phi(r_{\infty})\psi(m)$ (this makes sense because $\mathfrak{b}_{N}R\subseteq\delta_{N}$). 
\end{itemize} 
We say that two patching data $(\phi,\mathfrak{M},\psi)$ and $(\phi',\mathfrak{M}',\psi')$ of level $N$ are equivalent if $\phi=\phi'$ and there is an $R_{\infty}$- and $S_{\infty}$-equivariant isomorphism $\mathfrak{M}\cong\mathfrak{M}'$ that respects $\psi$ and $\psi'$. Since $R/\delta_{N}$ and $S_{\infty}/\mathfrak{b}_{N}$ are finite, one easily checks that for each $N\geq1$ there are only finitely many equivalence classes of patching data of level $N$. Furthermore, there is a reduction map from equivalence classes of level $N$ patching data to those of level $N-1$. One simply takes $\phi$ modulo $\delta_{N-1}$, sets the module to be $\mathfrak{M}/\mathfrak{b}_{N-1}$ and notes that $\psi$ induces a map $\mathfrak{M}/(\mathfrak{b}_{N-1}+\mathfrak{a}_{\infty})\rightarrow M/\mathfrak{b}_{N-1}$ because it is an $\Oe$-module map.

For each pair of integers $M\geq N\geq1$ we define a patching datum $D(M,N)$ of level $N$ as follows:
\begin{itemize}
\item the map $\phi$ is the composite $R_{\infty}\twoheadrightarrow R_{M}\hat{\otimes}\mathfrak{J}\twoheadrightarrow R\twoheadrightarrow R/\mathfrak{\delta}_{N}$ where the middle map is reduction mod $\mathfrak{a}_{\infty}$,

\item $\mathfrak{M}=M_{M}\hat{\otimes}\mathfrak{J}/\mathfrak{b}_{N}$. This works because $M_{M}\hat{\otimes}\mathfrak{J}$ is finite free over $S_{\infty}/\ker(p_{M})$ and $\mathfrak{b}_{N}\supseteq\mathfrak{b}_{M}\supseteq\ker(p_{M})$. 

\item The isomorphism $M_{M}\hat{\otimes}\mathfrak{J}/\mathfrak{a}_{\infty}\cong M$ induces $\psi:M_{M}\hat{\otimes}\mathfrak{J}/(\mathfrak{a}_{\infty}+\mathfrak{b}_{N})\cong M/\mathfrak{b}_{N}$ which is clearly compatible with $\phi$.
\end{itemize} 

It is important to observe that the reduction of $D(M,N)$ to level $N-1$ is equivalent to $D(M,N-1)$. Because of this, it is possible to find a sequence $(D(M_{N},N))_{N\geq1}$ such that the reduction of $D(M_{N,}N)$ to level $N-1$ is $D(M_{N-1},N-1)$. First, by finiteness we may choose an infinite sequence $(D(M_{i},1))_{i\geq1}$ at level $1$, all of whose elements are equivalent to one another. Define $D(M_{1},1)$ to be this equivalence class. From this sequence we may pick a further infinite subsequence $(D(M_{i'},2))_{i'\geq1}$ at level $2$ whose elements are all equivalent to one another; define $D(M_{2},2)$ to be this equivalence class. Proceeding in this way gives the construction.

Write $D(M_{N},N)=(\phi_{N},M_{M_{N}}\hat{\otimes}\mathfrak{J}/\mathfrak{b}_{N},\psi_{N})$. Then $\phi_{N-1}=\phi_{N}\mod\delta_{N-1}$ and for each $N$ we have an isomorphism
\begin{equation*}
\gamma_{N}:M_{M_{N}}\hat{\otimes}\mathfrak{J}/\mathfrak{b}_{N-1}\cong M_{M_{N-1}}\hat{\otimes}\mathfrak{J}/\mathfrak{b}_{N-1}
\end{equation*}
of $R_{\infty}\hat{\otimes}S_{\infty}$-modules together with commuting isomorphisms of both sides with $M/\mathfrak{b}_{N-1}$ compatible with $\phi_{N-1}$ and $\phi_{N}\mod\delta_{N-1}$. The key point is that the $S_{\infty}$-actions on either side were originally via the unrelated homomorphisms $S_{*}\rightarrow R_{*}$ but we have now patched them together, and similarly for the $R_{\infty}$-actions. Set $\phi_{\infty}=\underleftarrow{\lim}\,\phi_{N}:R_{\infty}\twoheadrightarrow R$. (The image is $R$ because $\delta_{N}$ are open and $\cap_{N}\delta_{N}=0$, and to see that it is surjective use the fact that any quotient of a complete local noetherian ring is complete.) Also set $M_{\infty}=\underleftarrow{\lim}_{N}M_{M_{N}}\hat{\otimes}\mathfrak{J}/\mathfrak{b}_{N}$ where the transition maps are the via the isomorphisms $\gamma_{N}$. Then $M_{\infty}$ is finite free over $S_{\infty}$ and $M_{\infty}/\mathfrak{b}_{N}\cong M_{M_{N}}\hat{\otimes}\mathfrak{J}/\mathfrak{b}_{N}$ as an $R_{\infty}\hat{\otimes}S_{\infty}$-module. Note that the action of $S_{\infty}$ on $M_{\infty}/\mathfrak{b}_{N}$ is via $S_{\infty}\rightarrow R_{M_{N}}\hat{\otimes}\mathfrak{J}\circlearrowleft M_{M_{N}}\hat{\otimes}\mathfrak{J}$ but the action of $R_{\infty}$ is via the surjection $R_{\infty}\twoheadrightarrow R_{M_{N}}\hat{\otimes}\mathfrak{J}$. So for each $N$ the action of $S_{\infty} on M_{\infty}/\mathfrak{b}_{N}$ factors through $R_{\infty}$. As $M_{\infty}=\underleftarrow{\lim}_{N}M_{\infty}/\mathfrak{b}_{N}$ it follows that the action of $S_{\infty}$ on $M_{\infty}$ factors through that of $R_{\infty}$. Since$S_{\infty}$ is a power series ring there is a local homomorphism $S_{\infty}\rightarrow R_{\infty}$ through which the action factors.

Finally we observe that $M_{\infty}/(\mathfrak{a}_{\infty}+\mathfrak{b}_{N})\cong M/\mathfrak{b}_{N}$ compatibly with $\phi_{\infty}\mod\delta_{N}$. Taking inverse limits it follows that $M_{\infty}/\mathfrak{a}_{\infty}\cong M$ compatibly with $\phi_{\infty}$, which also implies the final remaining statement that $\phi_{\infty}(\mathfrak{a}_{\infty}R_{\infty})\subseteq\mathrm{Ann}_{R}(M)$.  
\end{proof}

The following is a corollary to the proof of Theorem \ref{thm:main result} and will be used in Section \ref{sec:application}.
\begin{corollary}\label{cor:scalars}
Let $U,\bar{r}$ be as in Theorem \ref{thm:main result} and $\lam^{\vee}\in W^{\ord}(\bar{r})$. The mod $p$ Hecke operators $T_{\tv,j}$ for $v\in S_p$, $1\leq j\leq n$ each act by a nonzero scalar on $\Hom_{G(\ofp)}(F_{\lam}^{\vee},S(U^p,\F)[\m_{\bar{r}}^T])^{\ord}$. 
\end{corollary}

\begin{proof}
In the proof of Theorem \ref{thm:main result} we showed that $M$ is free over $\T:=\T_{\lam}^{T,\ord}(U)_{\m}$. We now adapt an argument from \cite{HLM}. We have $(M^{\d})_{\Qpbar}=\bigoplus_{\p} (M^{\d})_{\Qpbar}[\p_{\Qpbar}]$, where $\p$ runs over all minimal primes of $\T$. Since each automorphic representation contributing to $(M^{\d})_{\Qpbar}$ occurs with multiplicity $1$, and since they are unramified at places dividing $p$ by the choice of $U$, the characteristic $0$ Hecke operators $T^{(j)}_{\lam,\tv}$ for $v\in S_p$, $1\leq j\leq n$ each act on $(M^{\d})_{\Qpbar}[\p_{\Qpbar}]$ by a scalar in $\Oe^{\times}$. Hence they also act by a unit scalar on $(M^{\d})_{\Qpbar}/\p_{\Qpbar}\cong \Hom_{\Qpbar}(M_{\Qpbar}[\p_{\Qpbar}],\Qpbar)$. 

Now fix any choice of minimal prime $\p$. The desired statement is insensitive to extension of the field $E$, so we can assume that $\T/\p=\Oe$. Then $M/\p$ is a free $\Oe$-module, and it follows that $M/\p$ is a submodule of $M_{\Qpbar}/\p_{\Qpbar}$, so we deduce that the $T^{j}_{\lam,\tv}$ each act by a unit scalar on $M/\p$ as well, and it follows that they each act by a nonzero scalar on $(M/\p)_{\F}=M/\m$ and also $(M/\m)^{\vee}=\left(S^{\ord}(U,M_{\lam})\otimes_{\Oe}\F\right)[\m_{\bar{r}}^T]$. The result now follows from Proposition \ref{prop:hecke operators at p}.
\end{proof}

\subsection{Application}\label{sec:application}

In this section we complete the application to \cite{BH} mentioned in the introduction. Let $F$ and $G$ be as in Section \ref{subsec:aut forms} and assume $\bar{r}:G_F\rightarrow\GL_n(\F)$ and $U$ obey A1-A6 and B1-B3 of Section \ref{subsec:patching set-up} respectively. Define $T$ just as before the statement of Theorem \ref{thm:main result}. The assumption that $\bar{r}$ is modular and ordinary implies that $\bar{r}|_{G_{F_w}}$ is ordinary for each $w|p$ in $F$ in the sense of being conjugate to an upper-triangular representation (see Proposition 4.4.4 of \cite{BH}). In this situation, \cite{BH} constructs an admissible smooth representation $\Pi(\bar{r}|_{G_{F_w}})^{\ord}$ of $\GL_n(F_w)$ depending only on $\bar{r}|_{G_{F_w}}$ which is supposed to realize part of the (conjectural) mod $p$ local Langlands correspondence for $\bar{r}|_{G_{F_w}}$. The representation $\Pi(\bar{\rho})^{\ord}$ breaks up into indecomposable summands
\begin{equation*}
\Pi(\bar{\rho})^{\ord}= \bigoplus_{w\in W_{C_{\bar{\rho}}}} \Pi(\bar{\rho})_{C_{\bar{\rho}},w}
\end{equation*}
where $C_{\bar{\rho}}$ is a certain parabolic subgroup of $\GL_n$ and $W_{C_{\bar{\rho}}}$ an associated set of elements of the Weyl group (the details are not important here). Corollary 4.4.6 together with Proposition 3.4.5 of \emph{op. cit.} imply that for each $v\in S_p$, the set of Weyl group elements $w$ for $\bar{r}|_{G_{F_{\tv}}}$ is indexed by the representations $F_{\lambda_{\wtau}}^{\vee}$ appearing as a factor of a global Serre weight $F_{\lam}^{\vee}= \bigotimes_{\wtau\in\wI_p} F_{\lam_{\wtau}}^{\vee}\in W^{\ord}(\bar{r})$. We write $w_{\lambda_{\tv}}$ for the element corresponding to $F_{\lambda_{\tv}}$. Then one of the main results (Theorem 4.4.7) of \emph{op. cit.} is that there is an \emph{essential} injection of representations
\begin{equation}\label{eq:essential injection}
\bigoplus_{\lambda^{\vee} \in W^{\ord}(\bar{r})} \left( \bigotimes_{v\in S_p} \left( \Pi(\bar{r}|_{G_{F_{\tv}}})_{C_{\bar{r}|_{G_{F_{\tv}}}},w_{\lambda_{\tv}}}\otimes (\epsilon^{n-1}\circ\det)\right)\right)^{d_{\lambda}} \hookrightarrow S(U^p,\F)[\m^T_{\bar{r}}]^{\ord}
\end{equation}
for some multiplicities $d_{\lam}>0$, where the $\ord$ superscript on the right hand side denotes the largest subrepresentation of $S(U^p,\F)[\m^T_{\bar{r}}]$ whose irreducible subquotients are all subquotients of principal series.

\begin{remark}
Technically, \cite{BH} works with compact open subgroups $U$ satisfying a stronger definition of ``sufficiently small'' than the one we have used, but everything they do (at least leading up to Theorem 4.4.7) is still valid with our definition.
\end{remark}

\begin{theorem}\label{thm:application}
Suppose that $F,G,U,T$ and $\bar{r}$ obey the assumptions listed at the beginning of this section. Then the multiplicities $d_{\lambda}$ in (\ref{eq:essential injection}) are independent of the Serre weight $\lambda^{\vee}\in W^{\ord}(\bar{r})$. 
\end{theorem}

\begin{proof}
In the proof of (\ref{eq:essential injection}) in \cite{BH}, it is shown that 
\begin{equation*}
d_{\lambda}=\dim_{\F}\Hom_{G(\ofp)}\left(F_{\lambda}^{\vee},  
S(U^p,\F) [\m^T_{\bar{r}}]\right)[\eta]
\end{equation*}
where $\eta:\mathcal{H}(F_{\lam}^{\vee})\rightarrow\F$ is the unique ordinary character such that the eigenspace above is nonzero. However, Corollary \ref{cor:scalars} implies that 
\begin{equation*}
d_{\lambda}=\dim_{\F}\Hom_{G(\ofp)}\left(F_{\lambda}^{\vee},  
S(U^p,\F) [\m^T_{\bar{r}}]\right)^{\ord}
\end{equation*}
and now the result follows from Theorem \ref{thm:main result}.
\end{proof}

\begin{remark}
It should be possible to prove a result along the lines of Theorem 5.3.7 of \cite{HLM} saying that given a place $v\in S_p$ and $\rhobar:G_{F_v}\rightarrow \GL_n(\F)$, there exists a CM field $F$ and group $G$ as well as a globalization $\bar{r}:G_F\rightarrow \GL_n(\F)$ with $\rhobar=\bar{r}|_{G_{F_v}}$ to which Theorem \ref{thm:application} applies.
\end{remark}

\bibliographystyle{amsalpha}
\bibliography{ordinary_multiplicities}

\end{document}